\documentclass{article}
\usepackage{graphicx}
\usepackage[all]{xy} 
\usepackage{amsmath}
\usepackage{amsthm}
\usepackage{mathrsfs} 
\usepackage{array}  
\usepackage{amssymb}
\usepackage{calc}
\usepackage{verbatim}
\usepackage{txfonts}
\usepackage{textcomp}
\usepackage{colortbl}
\usepackage{tabularx} 
\usepackage[table]{xcolor}
\usepackage{lscape}
\usepackage{rotating}
\usepackage{tensor}
\usepackage{lmodern}
\usepackage{ytableau}

\usepackage{placeins} 
\usepackage[T1]{fontenc} 
\usepackage{titletoc}
\usepackage{pgf}
\usepackage{tikz}
\usepackage{young}
\usetikzlibrary{positioning}
\usetikzlibrary{matrix,arrows}
\usetikzlibrary{patterns}
\usetikzlibrary{shapes,decorations,shadows}

\DeclareMathOperator{\End}{End}

\DeclareMathOperator{\Hom}{Hom} 
\DeclareMathOperator{\rad}{rad} 
 
\DeclareMathOperator{\corank}{corank}

\DeclareMathOperator{\Rep}{Rep} 
 
\DeclareMathOperator{\rep}{rep}

\DeclareMathOperator{\GL}{GL}

\DeclareMathOperator{\Spec}{Spec} 
\definecolor{lblue}{rgb}{0.3,0.0,4.4}
\definecolor{lred}{rgb}{4.3,0.0,0.4}

\newcommand*{\punkte}{\dots\unkern}
\newcolumntype{C}[1]{>{\centering\arraybackslash}p{#1}}

\newcommand{\rar}{\rightarrow} 
\newcommand{\A}{\mathcal{A}} 
\newcommand{\B}{\mathcal{B}} 
 
\newcommand{\Pa}{\mathcal{P}} 
\newcommand{\Ca}{\mathcal{C}}

\newcommand{\Q}{\mathcal{Q}}

\newcommand{\N}{\mathcal{N}}

\newcommand{\df}{\underline{d}}

\newcommand{\Z}{\mathbf{Z}}

\newcommand{\C}{\mathbf{C}}
\newcommand{\Cal}{\mathcal{C}}
\newcommand{\res}{\mathrm{res}}
\newcommand{\sres}{\mathrm{sres}}

\newcommand{\dd}{\mathscr{D}}

\newcommand{\mc}{\mathcal}
\newcommand{\Add}{\mathscr{C}}

\newcommand{\g}{\mathfrak{g}}

\newcommand{\ds}{\dots}

\newcommand{\mbf}{\mathbf}
\newcommand{\vdim}{\mbf{v}}
\newcommand{\Stab}{\mathrm{Stab}}

\newcommand{\mf}{\mathfrak}

\newcommand{\wt}{\mathrm{wt}}

\newtheorem{theorem}{Theorem}[section]
\newtheorem{lemma}[theorem]{Lemma}
\newtheorem{definition}[theorem]{Definition}
\newtheorem{proposition}[theorem]{Proposition}
\newtheorem{corollary}[theorem]{Corollary}
\newtheorem{remark}[theorem]{Remark}

\newtheorem{example}[theorem]{Example}
\newtheorem{counterexample}[theorem]{Counter-example}

 \definecolor{lightblue}{rgb}{0.8,0.8,4.9}
  \definecolor{lightred}{rgb}{4.9,0.6,0.8}
\author{Gwyn Bellamy  \thanks{School of Mathematics and Statistics, University of Glasgow, 15 University Gardens, Glasgow, G12 8QW, United Kingdom. Gwyn.Bellamy@glasgow.ac.uk}~~ and~ Magdalena Boos \thanks{Ruhr-Universit\"at Bochum, Faculty of Mathematics,  D - 44780 Bochum, Germany. Magdalena.Boos-math@ruhr-uni-bochum.de}}

\makeatletter
\long\def\nnfoottext#1{\insert\footins{\footnotesize
    \interlinepenalty\interfootnotelinepenalty
    \splittopskip\footnotesep
    \splitmaxdepth \dp\strutbox \floatingpenalty \@MM
    \hsize\columnwidth \@parboxrestore
   \edef\@thefnmark{}
   \edef\@currentlabel{}\@makefntext
    {\rule{\z@}{\footnotesep}\ignorespaces
      #1\strut}}}
\makeatother

\begin{document}
\makeatletter
\let\@fnsymbol\@arabic
\makeatother
\parindent0pt
\title{\bf The (cyclic) enhanced nilpotent cone via quiver representations}


\date{}
\maketitle
 
 \begin{abstract}
The $\GL(V)$-orbits in the enhanced nilpotent cone $V\times \N(V)$ are (essentially) in bijection with the orbits of a certain parabolic $P\subseteq \GL(V)$ (the \textit{mirabolic} subgroup) in the nilpotent cone $\N(V)$. We give a new parameterization of the orbits in the enhanced nilpotent cone, in terms of representations of the underlying quiver.  This parameterization generalizes naturally to the enhanced cyclic nilpotent cone. Our parameterizations are different to the previous ones that have appeared in the literature. Explicit translations between the different parametrizations are given.
 \end{abstract}
 \setcounter{tocdepth}{2}
 \tableofcontents
\section{Introduction}\label{sect:intro}

Let $V$ be a complex finite-dimensional vector space and $\N(V) \subset \End(V)$ the nilpotent cone of $V$ i.e. the variety of nilpotent endomorphisms of $V$. The enhanced nilpotent cone is defined as $V\times \N(V)$; it admits a diagonal action of $\GL(V)$. This action has been examined in detail by several authors including Achar-Henderson \cite{AH}, Travkin \cite{Tr}, Mautner \cite{MautnerPaving} and Sun \cite{SunEnhancednilcone}. In particular the $\GL(V)$-orbits in $V\times \N(V)$ are enumerated; it was shown by  Achar-Henderson and Travkin that the orbits are naturally in bijection with bi-partitions of $\dim V$. The orbit closure relations are also described combinatorially.\\

This classification of $\GL(V)$-orbits in $V\times \N(V)$ was extended to the enhanced cyclic nilpotent cone by Johnson \cite{Joh}. Let $\Q(\ell)$ be the cyclic quiver with $\ell$ vertices, and $\Q_{\infty}(\ell)$ the framing $\infty \rightarrow 0$ of this quiver at the vertex $0$. The enhanced cyclic nilpotent cone $\N_{\infty}(\ell,x)$ is the space of representations of $\Q_{\infty}(\ell)$ with dimension one at the framing vertex $\infty$, and the endomorphism obtained by going once around the cycle having nilpotency $\le x$. The group $G = \prod_{i = 0}^{\ell-1} \GL_n$ acts on $\N_{\infty}(\ell,x)$ with finitely many orbits.\\

In this article, we return to the question of parameterizing the $G$-orbits in $\N_{\infty}(\ell,x)$. Our motivation comes from two quite different sources. Firstly, this is a generalization of the problem of studying parabolic conjugacy classes in the space of nilpotent endomorphisms; see section \ref{sec:parabolicintro}. Secondly, via the Fourier transform, the enhanced cyclic nilpotent cone plays a key role in the theory of admissible $\dd$-modules on the space of representations of $\Q_{\infty}(\ell)$. Applications of our parameterization of orbits to the representation theory of admissible $\dd$-modules are explored in the sister article \cite{BeB2}. 

\subsection{A representation theoretic parameterization}

There are (essentially) two different approaches to the classification of $G$-orbits in the (cyclic) enhanced nilpotent cone. The first is to consider it as a problem in linear algebra, that of classifying pairs $(v,X)$, of a nilpotent endomorphism $X$ and a vector $v \in V$, up to change of basis (``coloured endomorphisms and coloured vectors'' in the case of the enhanced cyclic nilpotent cone). This is the approach taken in \cite{AH}, \cite{Tr} and \cite{Joh}. 

Secondly, one can consider it as a problem in representation theory. Namely, it is clear that the enhanced cyclic nilpotent cone parameterizes representations of a particular algebra $\A_{\infty}(\ell,x)$ (realized as an admissible quotient of the corresponding path algebra), with the appropriate dimension vector. The usual enhanced nilpotent cone corresponds to $\ell = 1$. Then it is a matter of classifying the isomorphism classes of representations of this algebra. It is this latter approach that we take here. 

One natural way to try and classify the isomorphism classes of these algebras is to consider their universal covering algebras. This works well only if the representation type of the algebras $\A_{\infty}(\ell,x)$ is finite. Unfortunately, we show: 

\begin{proposition}\label{prop:typel}
The algebra $\A_{\infty}(\ell,x)$ is of finite representation type if and only if $\ell = 1$ and $x \le 3$. 
\end{proposition}

Moreover, we show that if $\ell >  1$ or $x > 3$, then for $(\ell,x) = (4,1), (2,2)$ the algebra $\A_{\infty}(\ell,x)$ is tame, and it is wild in all other case; see section \ref{sec:reptype}. Fortunately, the algebra $\A(\ell,x)$, whose representations correspond to the (non-enhanced) cyclic nilpotent cone, has finite representation type by Kempken \cite{Ke}. Hence we deduce its indecomposable representations from the universal covering algebra. From this we can read off the isomorphism classes of indecomposable representations $M$ of $\A_{\infty}(\ell,x)$  with $(\dim M)_{\infty} = 1$, even though $\A_{\infty}(\ell,x)$  does not have finite representation type in general. We introduce the set of \textit{Frobenius circle diagrams} $\Ca_F(\ell)$; from each Frobenius circle diagram one can easily reconstruct an indecomposable nilpotent representation of $\Q_{\infty}(\ell)$. Each Frobenius circle diagram $C$ has a weight $\wt_{\ell}(C)$. We also introduce the weight $\wt_{\ell}(\lambda)$ of a partition $\lambda$; see section \ref{sect:combi} for the definition of these combinatorial objects. 

\begin{theorem}\label{thm:paramainintro}
Fix $\ell, x \ge 1$. 
\begin{enumerate}
\item There are canonical bijections between:
\begin{itemize}
\item The set of isomorphism classes of indecomposable nilpotent representations $M$ of $\Q_{\infty}(\ell)$ with $(\dim M)_{\infty} = 1$. 
\item The set $\Ca_F(\ell)$ of Frobenius circle diagrams. 
\item The set of all partitions. 
\end{itemize}
\item These bijections restrict to bijections between:
\begin{itemize}
\item The set of isomorphism classes of indecomposable representations $M$ of $\A_{\infty}(\ell,x)$ with $(\dim M)_{\infty} = 1$. 
\item The set $\{ C \in \Ca_F(\ell) \ | \ \wt_{\ell}(C) \le x \}$ of Frobenius circle diagrams of weight at most $x$. 
\item The set $\{ \lambda \in \mc{P} \ | \ \wt_{\ell}(\lambda) \le x \}$ of partitions of weight at most $x$. 
\end{itemize}
\end{enumerate}
\end{theorem}

In the case of most interest to us, 
$$
\dim M = (1, n, \ds, n) = \varepsilon_{\infty} + n \delta
$$
where $\delta$ is the minimal imaginary root for the cyclic quiver, the classification can be interpreted combinatorially. If $\mc{P}$ denotes the set of all partitions, $\mc{P}_{\ell}$ the set of all $\ell$-multipartitions, then we show that:

\begin{corollary}\label{cor:paramcomb}
The $G$-orbits in the enhanced cyclic nilpotent cone $\N_{\infty}(\ell,n)$ are naturally labeled by the set 
$$
\mathcal{Q}(n,\ell) := \left\{ (\lambda;\nu) \in \mathcal{P} \times \mathcal{P}_{\ell} \ | \ \res_{\ell}(\lambda) + \sres_{\ell}(\nu) = n \delta\right\}. 
$$
\end{corollary}

Here $\res_{\ell}(\lambda)$ and $\sres_{\ell}(\nu)$ are the (shifted) $\ell$-residues of the corresponding partitions; see section \ref{sec:comborbits} for details. We note that our parameterization is clearly different from the parameterization given in \cite{AH}, \cite{Tr} and \cite{Joh}. In the case of the usual enhanced nilpotent cone ($\ell = 1$), we explain in subsection \ref{ssect:enc_transl} how to go between the two parameterizations (this is also explained in \cite[Lemma 2.4]{MautnerPaving}). When $\ell > 1$, we relate our parameterization to the one given by Johnson \cite{Joh} in subsection \ref{ssect:cenc_transl}. Corollary \ref{cor:paramcomb} also appeared recently in \cite[Remark 11.2.3]{DoGinT}.

\subsection{Parabolic conjugacy classes in the nilpotent cone}\label{sec:parabolicintro}

In \cite{B2}, the second author considered the adjoint action of parabolic subgroups $P\subseteq\GL(V)$ on the varieties $\N^{(x)}$ of $x$-nilpotent endomorphisms of $V$. In particular, the question of which pairs $(P,x)$ have the property that there are finitely many $P$-orbits on $\N^{(x)}$ is addressed. The methods used in \textit{loc. cit.} are mainly representation-theoretic: the algebraic group action is translated, via an associated fibre bundle, to a setup of representations of a certain finite quiver with relations. In all those cases (excluding the enhanced nilpotent cone) where there are finitely many $P$-orbits, the orbits are enumerated and the orbit closures are described in detail. \\[1ex]

In this article, we describe how $\GL(V)$-orbits in the enhanced nilpotent cone relate to the $P$-orbits of a particular parabolic (the ``mirabolic'') subgroup on the nilpotent cone $\N$. More generally, for any dimension vector $\mbf{d}$ of the cyclic quiver, there is a certain parabolic $P \subset \GL_{\mbf{d}}$ such that $\GL_{\mbf{d}}$-orbits in the cyclic enhanced nilpotent cone $\N_{\infty}(\ell,x,\mbf{d})$ are in bijection with $P$-orbits in the cyclic nilpotent cone $\N(\ell,x,\mbf{d})$. This can be seen as a first step in the generalization of the above question to the case of parabolic conjugacy classes in the nilcone of Vinberg's $\theta$-representations. See remark \ref{rem:theta} for more details.

\subsection{Admissible $\dd$-modules}\label{sec:addintro} 

As mentioned previously, one motivation for developing a quiver-theoretic approach to the $G$-orbits in the enhanced cyclic nilpotent cone is that one gets in this way immediate results regarding the category of admissible $\dd$-modules on the space $X = \Rep(\Q_{\infty}(\ell); \mathbf{v})$ of representations of the framed cyclic quiver. Fix a character $\chi$ of the Lie algebra $\mf{g}$ of $G$. The category $\Add_{\chi}$ of admissible $\dd$-modules on $X$ is the category of all smooth $(G,\chi)$-monodromic $\dd$-modules on $X$, whose singular support lies in a certain Lagrangian $\Lambda$. Essentially those modules whose singular support is nilpotent in the conormal direction; see \cite{BeB2} for details. Admissible $\dd$-modules are always regular holonomic, and it is easily shown (since $\N_{\infty}(\ell,n)$ has finitely many $G$-orbits) that there are only finitely many simple objects in $\Add_{\chi}$. The behaviour of the category $\Add_{\chi}$ depends heavily on the parameter $\chi$.\\ 

Using the results of this article, we are able to describe precisely, in \cite{BeB2}, the locus where $\Add_{\chi}$ is semi-simple. It is shown that this is the complement to countably many (explicit) affine hyperplanes. In \cite{BeB2}, we are able to list 10 other properties of the category $\Add_{\chi}$ are equivalent to ``$\Add_{\chi}$ is semi-simple". The reason why our new parametrization of the $G$-orbits in the cyclic enhanced nilpotent cone is so useful in this context is because it allows us to easily compute the fundamental group of the orbits. 

\subsection{Outline of the article}

In section two the required background results in representation theory are given. Section three introduces the combinatorial notions that we will use, in particular the notion of Frobenius circle diagrams. In section four we consider parabolic conjugacy classes in the nilpotent cone. Section five deals with the parameterization of orbits in the enhanced nilpotent cone (i.e. for $\ell = 1$). Then the enhanced cyclic nilpotent cone is considered in section six. In particular, we prove Proposition \ref{prop:typel}, Theorem \ref{thm:paramainintro} and Corollary \ref{cor:paramcomb}. 

\subsection{Outlook}
Our representation-theoretic approach makes it possible to apply techniques from representation theory to better understand the geometry of the enhanced cyclic nilpotent cone. For example, there are several techniques available to calculate degenerations; that is, orbit closure relations. Namely, the results of Zwara \cite{Zw1,Zw2} and Bongartz \cite{Bo1,Bo2} are applicable. By making use of these, we hope to define combinatorially the closure ordering on the set $\mc{Q}(n,\ell)$ in the near future.\\

{\bf Acknowledgements:} The authors would like to thank C. Johnson for his very precise and valuable ideas regarding the translation between our parametrization and his original parametrization of orbits. We also thank K. Bongartz and M. Reineke for helpful remarks on the subject. The first author was partially supported by EPSRC grant EP/N005058/1.
\section{Theoretical background}\label{sect:theory}
Let $\C$ be the field of complex numbers and $\GL_n\coloneqq\GL_n(\C)$ the general linear group, for a fixed integer $n\in\textbf{N}$, regarded as an affine algebraic group. 

\subsection{Quiver representations}

The concepts in this subsection are explained in detail in \cite{ASS}. A \textit{(finite) quiver} $\Q$ is a direct graph $\Q=(\Q_0,\Q_1,s,t)$, with $\Q_0$ a finite set of \textit{vertices}, and $\Q_1$ a finite set of \textit{arrows}, with $\alpha\colon s(\alpha)\rightarrow t(\alpha)$.
The \textit{path algebra} $\C\Q$ is the $\C$-vector space with basis consisting of all paths in $\Q$, that is, sequences of arrows $\omega=\alpha_s\punkte\alpha_1$, such that $t(\alpha_{k})=s(\alpha_{k+1})$ for all $k\in\{1,\punkte,s-1\}$; we formally include a path $\varepsilon_i$ of length zero for each $i\in \Q_0$ starting and ending in $i$. The multiplication $\omega\cdot\omega'$ of two paths $\omega= \alpha_s ... \alpha_1$ and $\omega' = \beta_t ... \beta_1$ is by concatenation if $t(\beta_t)=s(\alpha_1)$, and is zero otherwise. This way, $\C\Q$ becomes an associative $\C$-algebra. The \textit{path ideal} $I(\C\Q)$ of $\C \Q$ is the (two-sided) ideal generated by all paths of positive length; then an arbitrary ideal $I$ of $\C \Q$ is called \textit{admissible} if there exists an integer $s$ with $I(\C \Q)^s\subset I\subset I(\C \Q)^2$.\\[1ex]

A finite-dimensional $\C$-representation of $\Q$ is a tuple \[((M_i)_{i\in \Q_0},(M_\alpha\colon M_i\rightarrow M_j)_{(\alpha\colon i\rightarrow j)\in \Q_1}),\] of $\C$-vector spaces $M_i$ and $\C$-linear maps $M_{\alpha}$. There is the natural notion of a \textit{morphism of representations} $M=((M_i)_{i\in \Q_0},(M_\alpha)_{\alpha\in \Q_1})$ and
 \mbox{$M'=((M'_i)_{i\in \Q_0},(M'_\alpha)_{\alpha\in \Q_1})$}, which is defined to be a tuple of $\C$-linear maps $(f_i\colon M_i\rightarrow M'_i)_{i\in \Q_0}$, such that $f_jM_\alpha=M'_\alpha f_i$ for every arrow $\alpha\colon i\rightarrow j$ in $\Q_1$. For a representation $M$ and a path $\omega$ in $\Q$ as above, we denote $M_\omega=M_{\alpha_s}\cdot\punkte\cdot M_{\alpha_1}$. A representation $M$ is said to be \textit{bound by $I$} if $\sum_\omega\lambda_\omega M_\omega=0$ whenever $\sum_\omega\lambda_\omega\omega\in I$. Thus, we obtain certain categories: the abelian $\C$-linear category $\rep_{\C} \C\Q$ of all representations of $\Q$  and the category $\rep_{\C} \C \Q/I$ of representations of $\Q$ bound by $I$; the latter is equivalent to the category of finite-dimensional $\A$-representations, where $\A\coloneqq \C \Q/I$ is the quotient algebra.\\[1ex]

Given a representation $M$ of $\Q$, its \textit{dimension vector} $\dim M\in\mathbf{N}\Q_0$ is defined by $(\dim M)_{i}=\dim_{\C} M_i$ for $i\in \Q_0$. Fixing a dimension vector $\mbf{d}\in\mathbf{N}\Q_0$, we obtain a  full subcategory  $\rep_{\C} \A(\mbf{d})$ of $\rep_{\C} \A$ which consists of representations of dimension vector $\mbf{d}$. Let $\Rep(\Q,\mbf{d}):= \bigoplus_{\alpha\colon i\rightarrow j}\Hom_{\C}(\C^{d_i},\C^{d_j})$; points of $\Rep(\Q,\mbf{d})$ correspond to representations $M\in\rep_{\C} \C \Q(\df)$ with $M_i=\C^{d_i}$ for $i\in \Q_0$.  Via this correspondence, the set of such representations bound by $I$ corresponds to a closed subvariety $\Rep(\C \Q / I,\mbf{d})\subset \Rep(\Q,\mbf{d})$. The set $\mbf{N} \Q_0$ of dimension vectors is partially ordered by $\alpha \ge \beta$ if $\alpha_i \ge \beta_i$ for all $i$ and we say that $\alpha > \beta$ if $\alpha \ge \beta$ with $\alpha \neq \beta$. A dimension vector $\alpha$ is called \emph{sincere} if $\alpha_i > 0$ for all $i$. The algebraic group $\GL_{\mbf{d}}=\prod_{i\in \Q_0}\GL_{d_i}$ acts on $\Rep(\Q,\mbf{d})$ and on $\Rep(\C \Q / I,\mbf{d})$ via base change. The $\GL_{\mbf{d}}$-orbits of this action are in bijection with the isomorphism classes of representations $M$ in $\rep_{\C} \A(\mbf{d})$.\\[1ex]

The Krull-Remak-Schmidt Theorem says that every finite-dimensional $\A$-representation decomposes into a direct sum of indecomposable representations. We denote by  $\Gamma_{\A}=\Gamma(\Q,I)$ the \textit{Auslander-Reiten quiver} of $\rep_{\C}\A$. 

\subsection{Representation types}\label{ssect:repTypes}

Consider a finite-dimensional, basic $\C$-algebra $\A:=\C \Q/I$. The algebra $\A$ is said to be of \textit{finite representation type} if there are only finitely many isomorphism classes of indecomposable representations. If it is not of finite representation type, the algebra is of \textit{infinite representation type}. The Dichotomy Theorem of Drozd \cite{Dr} says that if $\A$ is of infinite type, then $\A$ is one of two type: 
 \begin{itemize}
 \item  \textit{tame representation type} (or \textit{is tame}) if, for every integer $n$, there is an integer $k$ and finitely generated $\C[x]$-$\A$-bimodules $M_1,\punkte,M_{k}$ which are free over $\C [x]$, such that for all but finitely many isomorphism classes of indecomposable right $\A$-modules $M$ of dimension $n$, there are elements $i\in\{1,\punkte,k\}$ and $\lambda\in \C$, such that  $M\cong  \C[x]/(x-\lambda)\otimes_{\C[x]}M_i$.
 \item \textit{wild representation type} (or \textit{is wild}) if there is a finitely generated $\C \langle X,Y\rangle$-$\A$-bimodule $M$ which is free over $\C\langle X,Y\rangle$ and sends non-isomorphic finite-dimensional indecomposable $\C \langle X,Y\rangle$-modules via the functor $\_\otimes_{\C \langle X,Y\rangle}M$ to non-isomorphic indecomposable $\A$-modules.
\end{itemize} 

If $\A$ is a tame algebra then there are at most one-parameter families of pairwise non-isomorphic indecomposable $\A$-modules; in the wild case there are families of representations of arbitrary dimension.\\[1ex]

Several different criteria are available to determine the representation type of an algebra. We say that an algebra $\B = \C \Q'/I'$ is a \textit{full subcategory} of $\A = \C \Q/I$, if $\Q'$ is a \textit{convex subquiver} of $\Q$ (that is, a path closed full subquiver) and $I'$ is the restriction of $I$ to $\C \Q'$. \\[1ex]
An indecomposable projective $P$ has \textit{separated radical} if, for any two non-isomorphic direct summands of its radical, their supports (as subsets of $\Q$) are disjoint. We say that $\A$  \textit{fulfills the separation condition} if every projective indecomposable has a separated radical. \\[1ex]
In general, the definition of a \textit{strongly simply connected} algebra is quite involved. However, in case of a triangular algebra $\A$ (meaning that the corresponding quiver $\Q$ has no oriented cycles) there is an equivalent description: $\A$ is \textit{strongly simply connected} if and only if every convex subcategory of $\A$ satisfies the separation condition \cite{Sko2}. 

For a triangular algebra $\A = \C \Q/I$, the \textit{Tits form} $q_{\A}:\mathbf{Z}^{\Q_0}\rightarrow \mathbf{Z}$  is the integral quadratic form defined by 
\[q_{\A}(v) = \sum_{i\in\Q_0} v_i^2 - \sum_{\alpha:i\rightarrow j\in\Q_1} v_iv_j + \sum_{i,j\in\Q_0} r(i,j)v_iv_j;\]
for $v=(v_i)_i\in \mathbf{Z}^{\Q_0}$; here $r(i,j) := \dim \varepsilon_i R \varepsilon_j$, for any minimal generating subspace $R$ of $I$. \\[1ex]
The quadratic form $q_{\A}$ is called \textit{weakly positive}, if $q_{\A}(v) > 0$ for every $v\in\mathbf{N}^{\Q_0}$; and \textit{(weakly) non-negative}, if $q_{\A}(v) \geq 0$ for every $v\in\mathbf{Z}^{\Q_0}$ (or $v\in\mathbf{N}^{\Q_0}$, respectively). These concepts are closely related to the representation type of $\A$ and many results are, for example, summarized by De la Pe\~na and Skowro\'{n}ski in \cite{DlPS}. There are many necessary and sufficient criteria for finite, tame and wild types available, for example by Bongartz \cite{Bo4} and Br\"ustle, De la Pe\~na and Skowro{\'n}ski  \cite{BdlPS}. For our purposes, however, the following statement, which follows from these results, suffices.

\begin{lemma}\label{lem:wild_crit}
 Let $\A$ be strongly simply connected. $\A$ is of wild representation type if and only if there exists $v\in \mathbf{N}^{\Q_0}$, such that $q_{\A}(v)\leq -1$.
\end{lemma}

\subsection{Group actions}

If the algebraic group $G$ acts on an affine variety $X$, then $X/G$ denotes the set of orbits and $X/\!/ G := \Spec \C[X]^G$ is the categorical quotient. The following is a well-known fact on associated fibre bundles \cite{Se}, which will help translating certain group actions.
\begin{lemma}\label{thm:basis_transl}
Let $G$ be an algebraic group, let $X$ and $Y$ be $G$-varieties, and let $\pi : X \rar Y$ be a $G$-equivariant morphism. Assume that $Y$ is a single $G$-orbit, $Y = G \cdot y_0$. Let $H$ be the stabilizer of $y_0$ and set $F:= \pi^{-1} (y_0)$. Then $X$ is isomorphic to the associated fibre bundle $G\times^HF$, and the embedding $\phi: F \hookrightarrow X$ induces a bijection between $H$-orbits in $F$ and $G$-orbits in $X$ preserving orbit closures.
\end{lemma}

For an element $x \in \g := \mathrm{Lie} \ G$, its centralizer in $G$ is denoted $Z_G(x)$, and its centralizer in $\g$ is $Z_{\mf{g}}(x)$.

\section{Combinatorial objects}\label{sect:combi}

In this subsection, we define the combinatorial objects we use later.

\subsection{(Frobenius) Partitions and Young diagrams}\label{ssect:part_YD}

The set of all weakly decreasing partitions is denoted $\mathcal{P}$ and $\mathcal{P}_{\ell}$ denotes the set of all $\ell$-multipartitions. The subset of $\mc{P}$, resp. of $\mc{P}_{\ell}$, consisting of all partitions of $n\in\mathbf{N}$, resp. of all $\ell$-multipartitions of $n$, is denoted $\mc{P}(n)$, resp. $\mc{P}_{\ell}(n)$. Then $\Pa_2(n)$ is the set of \textit{bipartitions} of $n$, that is, of tuples of partitions $(\lambda,\mu)$, such that $\lambda\vdash m$ and $\mu\vdash n-m$ for some integer $m\leq n$.Given a partition $\lambda$, its \textit{Young diagram} is denoted by $Y(\lambda)$. \\[1ex]
The transpose of the partition $\lambda$ is denoted $\lambda^t$ and we define $s(\lambda)$ to be the number of diagonal entries of $Y(\lambda)$, that is, $s(\lambda)=\sharp\{i\mid \lambda_i\geq i\}$. 

\begin{definition}
We denote by $\Pa_F(n)$ the set of \textit{Frobenius partitions} of $n$. That is, the set pairs of tuples of strictly decreasing integers $(a_1 > \cdots > a_k \ge 0 )$ and $(b_1 > \cdots > b_k \ge 0)$ such that $\sum_{i=1}^k (a_i+b_i+1) = n$. We call $k$ the \textit{length} of $(\mbf{a},\mbf{b})$.
\end{definition}

It is a classical result that the set of Frobenius partitions $\Pa_F(n)$ can be naturally identified with the set $\mc{P}(n)$ of partitions of $n$. To be explicit, this is a bijection $\varphi: \Pa(n)\rightarrow \Pa_F(n)$ defined as follows:\\[1ex]
For a partition $\lambda\vdash n$, let $\varphi(\lambda)\coloneqq (\mbf{a}(\lambda),\mbf{b}(\lambda))$, where $\mbf{a}(\lambda)_i:=\lambda^t_i-i$ and $\mbf{b}(\lambda)_i:=\lambda_i-i$ for $i\leq s(\lambda)$. Graphically speaking, the Frobenius partition can be read off the Young diagram $Y(\lambda)$: $a_i$ is the number of boxes below the $i$th diagonal and $b_i$ the number of boxes to the right of the $i$th diagonal.  \\[1ex]

We also associate to a Frobenius partition $(\mbf{a},\mbf{b})$ the strictly decreasing partition 
$$
P(\mbf{a},\mbf{b}) := (p_1,...,p_k), \quad \textrm{where} \quad p_i:=a_i+b_i+1.
$$
One cannot recover $(\mbf{a},\mbf{b})$ from $P(\mbf{a},\mbf{b})$ in general. 

\begin{example}
Let $(\mbf{a},\mbf{b})=((4,2,0),(6,3,0))$ be a Frobenius partition. Then $P(\mbf{a},\mbf{b})=(11,6,1)$. 
 \begin{center}
 \begin{ytableau}

&&&&&&&&&&\\
&&&&&\\
\\
\end{ytableau}
\end{center}
Moreover, $\varphi(7,5,3,2,1) = ((4,2,0),(6,3,0))$ which we naturally get by pigeon-holing the Frobenius partition into the partition: 
 \begin{center}
 \begin{ytableau}
s_1&*(lightblue)1&*(lightblue)2&
*(lightblue)3&*(lightblue)4&*(lightblue)5&
*(lightblue)6\\
*(lightred)4&s_2&*(lightblue)1&*(lightblue)2
&*(lightblue)3\\
*(lightred)3&*(lightred)2& s_3\\
*(lightred)2&*(lightred)1\\
*(lightred)1\\
\end{ytableau}
\end{center}
\end{example}

\subsubsection{Weights}

If $\lambda = (\lambda_1 \ge \cdots \ge \lambda_k > 0)$ is a non-trivial partition, define 
$$
\wt_{\ell}(\lambda) :=  | \{ -k < i < \lambda_1  \ | \ i \equiv 0 \ \mathrm{mod} \ \ell \}|
$$
to be the \textit{$\ell$-weight} of $\lambda$. If $\varphi(\lambda) = (a_1 > \cdots ; b_1 > \cdots )$ is its Frobenius form, then $\wt_{\ell}(\lambda)$ equals $| \{ -b_1 \le  i \le a_1  \ | \ i \equiv 0 \ \mathrm{mod} \ \ell \}|$. Pictorially, one simply counts the number of boxes of content $0 \ \mathrm{mod} \ \ell$ in the first Frobenius hook of $\lambda$.   

\subsection{The affine root system of type $\mathsf{A}$}\label{sec:affineA}

Throughout the article, $\Q(\ell)$ will denote the cyclic quiver with $\ell$ vertices, whose underlying graph is the Dynkin diagram of type $\widetilde{\mathsf{A}}_{\ell-1}$. Then, as explained in the introduction, $\Q_{\infty}(\ell)$ is the framed cyclic quiver. We denote by $R \subset \Z {\Q(\ell)_0}$ the set of \textit{roots} and $R^+ = R \cap \mbf{N} \Q(\ell)_0$ the subset of \textit{positive roots}. If $\delta = (1, \ds, 1)$ denotes the minimal imaginary root and $\Phi := \{ \alpha \in R \ | \ \varepsilon_0 \cdot \alpha = 0 \}$ is the finite root system of type $\mathsf{A}_{\ell-1}$, then 
$$
R = \{ n \delta + \alpha \ |  \ n \in \Z, \ \alpha \in \Phi \cup \{ 0 \} \} \smallsetminus \{ 0 \}. 
$$
Let $\lambda$ be a partition. Recall that the content $\mathrm{ct}(\Box)$ of the box $\Box \in Y(\lambda)$ in position $(i,j)$ is the integer $j - i$. We fix a generator $\sigma$ of the cyclic group $\Z_{\ell}$. Given a partition $\lambda$, the $\ell$-residue of $\lambda$ is defined to be the element $\res_{\ell}(\lambda) := \sum_{\Box \in \lambda} \sigma^{\mathrm{ct}(\Box)}$ in the group algebra $\Z[\Z_{\ell}]$. Similarly, given an $\ell$-multipartition $\nu$, the shifted $\ell$-residue of $\nu$ is defined to be 
$$
\sres_{\ell}(\nu) = \sum_{i = 0}^{\ell-1} \sigma^i \res_{\ell}(\nu^{(i)}). 
$$
We identify the root lattice $\Z \Q(\ell)$ of $\Q(\ell)$ with $\Z \{ \Z_{\ell} \}$ by $\varepsilon_i \mapsto \sigma^i$.

\subsection{Matrices}

We fix the nilpotent cone $\N$ of nilpotent matrices in $\mf{gl}(V)$. The $\GL(V)$-conjugacy classes in $\N$ are labeled by their Jordan normal forms. In order to make use of these concepts later, we fix some notation here. We denote by $J_k$ the nilpotent Jordan block of size $k$; and by  $J_{\lambda}$ the nilpotent matrix $J_{\lambda_1} \oplus \cdots \oplus J_{\lambda_k}$ in Jordan normal form of block sizes $\lambda_1 \ge \cdots \ge \lambda_k$. 

\subsection{Circle diagrams}\label{sec:circle}

Given a positive integer $\ell > 0$, we define a \textit{circle diagram} of type $\ell$ to be a quiver $C$, whose set $C_0$ of vertices is partitioned into $\ell$ blocks $b_0,..., b_{\ell-1}$, such that each vertex has at most one incoming arrow and one outgoing arrow; and an arrow can only be drawn from a vertex in block $b_i$ to a vertex in block $b_{i+1}$; or from a vertex in block $b_{\ell-1}$ to a vertex in block $b_{0}$, and there are no oriented cycles. We say that the vertices in block $b_i$ are \textit{in position $i$} and call the vector $\mathbf{d}= (d_0,\dots,d_{\ell-1})$, where $d_i$ is the number of vertices in position $i$, the \textit{dimension vector} of $C$. Given a circle diagram, each complete connected path of arrows is called a \textit{circle}. The number of arrows in a circle is its \textit{length}.


\begin{example}
A circle diagram of dimension vector $(2,3,3,2)$ is given by

\begin{center}
\scalebox{0.6}{
\begin{tikzpicture}[->,>=stealth',shorten >=1pt,auto,node distance=3cm,
  thick,
  node/.style={}]
  \node (1) {$\bullet$};
\node (1') [left=0.25 of 1] {$\bullet$};
\node (1'') [left=0.15 of 1'] {\textbf{$b_0$}};
\node (2) [above right=0.5 and 0.5 of 1] {$\bullet$};
\node (2') [above=0.25 of 2] {$\bullet$};
\node (2'') [above=0.25 of 2'] {$\bullet$};
\node (2''') [above=0.15 of 2''] {\textbf{$b_1$}};
\node (3) [below right=0.5 and 0.5 of 2] {$\bullet$};
\node (3') [right=0.25 of 3] {$\bullet$};
\node (3'') [right=0.25 of 3'] {$\bullet$};
\node (3''') [right=0.15 of 3''] {\textbf{$b_2$}};
\node (4) [below left=0.5 and 0.5 of 3] {$\bullet$};
\node (4') [below=0.25 of 4] {$\bullet$};
\node (4'') [below=0.15 of 4'] {\textbf{$b_3$}};
\path[->]
(1) edge [bend left=15]  (2')
(2) edge [bend left=15]  (3)
(3) edge [bend left=15]  (4)
(4) edge [bend left=15]  (1)
(3') edge [bend left=15]  (4')
(4') edge [bend left=15]  (1')
(2'') edge [bend left=15]  (3'');
\end{tikzpicture}}
\end{center}
This circle diagram has one circle of length $1$, one circle of length $2$ and one circle of length $4$.
\end{example}
The set of all circle diagrams of type $\ell$, modulo permutation of vertices in the same position, is denoted $\Ca(\ell)$. The subset consisting of all circle diagrams, whose circles have length at most $x$, is denoted $\Ca^{(x)}(\ell)$. Furthermore, we denote by $\ell(C)$ the length of a circle diagram $C$, that is, the number of circles in the diagram.\\[1ex]

A \textit{Frobenius circle diagram} is a circle diagram $C$ of type $\ell$, with $t$ circles, such that: 
\begin{enumerate}
\item Each circle $C(i)$ contains a distingushed (or \textit{marked}) vertex $s_i$ in position $0$. 
\item If $a_i$ is the number of vertices following $s_i$ in the circle, and $b_i$ the number of vertices preceeding $s_i$ in the circle, then, after possibly relabelling circles, 
$$
\mbf{a}=(a_{1},\ds,a_{t}), \quad \mbf{b}=(b_1,\ds,b_{t})
$$
determine a Frobenius partition.
\end{enumerate}
The set of Frobenius circle diagrams is denoted $\Ca_F(\ell)$.

\begin{example}
A Frobenius circle diagram of $(4,5,4,3)$ is given by

\begin{center}
\scalebox{0.6}{
\begin{tikzpicture}[->,>=stealth',shorten >=1pt,auto,node distance=3cm,
  thick,
  node/.style={}]
\node (1) {\fbox{$s_1$}};
\node (1') [left=0.25 of 1] {$\bullet$};
\node (1'') [left=0.25 of 1'] {\fbox{$s_2$}};
\node (1''') [left=0.25 of 1''] {\fbox{$s_3$}};
\node (2) [above right=0.5 and 0.5 of 1] {$\bullet_1$};
\node (2') [above=0.25 of 2] {$\bullet$};
\node (2'') [above=0.25 of 2'] {$\bullet$};
\node (2''') [above=0.25 of 2''] {$\bullet$};
\node (2'''') [above=0.25 of 2'''] {$\bullet$};
\node (3) [below right=0.5 and 0.5 of 2] {$\bullet_2$};
\node (3') [right=0.25 of 3] {$\bullet$};
\node (3'') [right=0.25 of 3'] {$\bullet$};
\node (3''') [right=0.25 of 3''] {$\bullet$};
\node (4) [below left=0.5 and 0.5 of 3] {$\bullet_{3}$};
\node (4') [below=0.25 of 4] {$\bullet$};
\node (4'') [below=0.25 of 4'] {$\bullet$};
\path[->]
(1) edge [bend left=15]  (2')
(2) edge [bend left=15]  (3)
(3) edge [bend left=15]  (4)
(4) edge [bend left=15]  (1)
(2') edge [bend left=15]  (3')
(3') edge [bend left=15]  (4')
(3'') edge [bend left=15]  (4'')
(4') edge [bend left=15]  (1')
(1') edge [bend left=15]  (2'')
(1''') edge [bend left=15]  (2'''')
(1'') edge [bend left=15]  (2''')
(2''') edge [bend left=15]  (3''')
(4'') edge [bend left=15]  (1'');
\end{tikzpicture}}
\end{center}
Then there are three circles: One circle of length $8$ with mark $s_1=4$, one circle of length $4$ with mark $s_2=3$ and one circle of length $1$ with mark $s_3=1$. A Frobenius partition arises as $((3,2,0),(5,2,1))$ which can be visualized by the partition $(6,4,4,2)$; and by the diagrams
\begin{center}\small\begin{tikzpicture}
\matrix (m) [matrix of math nodes, row sep=0.8em,
column sep=0.8em, text height=0.6ex, text depth=0.2ex]
{ \bullet & \bullet & \bullet & \bullet &\bullet & \bullet\\
 \bullet & \bullet & \bullet & \bullet & &  \\
 \bullet & \bullet & \bullet & \bullet &&   \\
 \bullet & \bullet &  & & &   \\
 };
\path[->]
(m-1-1) edge  (m-1-2)
(m-1-2) edge  (m-1-3)
(m-1-3) edge  (m-1-4)
(m-1-4) edge  (m-1-5)
(m-1-5) edge  (m-1-6)
(m-2-2) edge  (m-2-3)
(m-2-3) edge  (m-2-4)
(m-3-3) edge  (m-3-4)
(m-4-1) edge  (m-3-1)
(m-3-1) edge  (m-2-1)
(m-2-1) edge  (m-1-1)
(m-4-2) edge  (m-3-2)
(m-3-2) edge  (m-2-2)
;

\draw [blue!10!white,fill=blue!10!white] (3,-1) -- (3,1) -- (6,1) -- (6,0.5) -- (3.5,0.5) -- (3.5,-1) -- (3,-1);
\draw [blue!10!white,fill=blue!10!white] (4,-0.5) -- (4,0) -- (5,0) -- (5,-0.5) -- (4,-0.5);


\draw (3,0) --(3,-0.5) -- (3,-1) -- (3.5,-1) -- (4,-1) --(4,-0.5) -- (5,-0.5) -- (5,0.5) -- (6,0.5) -- (6,1) -- (3,1) -- (3,0);
\draw (3,0) --(5,0);
\draw (3,-0.5) -- (4,-0.5);
\draw (3,0.5) -- (5,0.5);
\draw (3.5,1) -- (3.5,-1);
\draw (4,1) --(4,-0.5);
\draw (4.5,1) -- (4.5,-0.5);
\draw (5,1) --(5,0.5);
\draw (5.5,1) -- (5.5,0.5);

\node at (2,0) {$=$};
\node at (3.25,0.75) {$s_1$};
\node at (3.75,0.25) {$s_2$};
\node at (4.25,-0.25) {$s_3$};

\end{tikzpicture} 

\end{center} 
\end{example}

Clearly, not every circle diagram with arbitrary marks yields a Frobenius circle diagram, as the following counterexample shows.
\begin{counterexample}
Consider the circle diagram with marks $s_1, s_2$ and $s_3$:
\begin{center}
\scalebox{0.6}{
\begin{tikzpicture}[->,>=stealth',shorten >=1pt,auto,node distance=3cm,
  thick,
  node/.style={}]
\node (1) {$\bullet$};
\node (1') [left=0.25 of 1] {\fbox{$s_1$}};
\node (1'') [left=0.25 of 1'] {\fbox{$s_2$}};
\node (1''') [left=0.25 of 1''] {\fbox{$s_3$}};
\node (2) [above right=0.5 and 0.5 of 1] {$\bullet_1$};
\node (2') [above=0.25 of 2] {$\bullet$};
\node (2'') [above=0.25 of 2'] {$\bullet$};
\node (2''') [above=0.25 of 2''] {$\bullet$};
\node (3) [below right=0.5 and 0.5 of 2] {$\bullet_2$};
\node (3') [right=0.25 of 3] {$\bullet$};
\node (3'') [right=0.25 of 3'] {$\bullet$};
\node (3''') [right=0.25 of 3''] {$\bullet$};
\node (4) [below left=0.5 and 0.5 of 3] {$\bullet_{3}$};
\node (4') [below=0.25 of 4] {$\bullet$};
\node (4'') [below=0.25 of 4'] {$\bullet$};
\path[->]
(1) edge [bend left=15]  (2)
(2) edge [bend left=15]  (3)
(3) edge [bend left=15]  (4)
(4) edge [bend left=15]  (1')
(1') edge [bend left=15]  (2')
(3') edge [bend left=15]  (4')
(4') edge [bend left=15]  (1'')
(1'') edge [bend left=15]  (2'')
(3'') edge [bend left=15]  (4'')
(4'') edge [bend left=15]  (1''')
(2''') edge [bend left=15]  (3''')
(1''') edge [bend left=15]  (2''')
;\end{tikzpicture}}
\end{center}
This does not correspond to any Frobenius partition.
\end{counterexample}

By definition, each Frobenius circle diagram gives rise to a partition. Conversely, if a partition $(\mbf{a},\mbf{b})$ is given in Frobenius form, then for each Frobenius hook $(a_i,b_i)$, we construct a circle $C(i)$ whose vertices are in bijection with the boxes of the hook, a vertex $u$ being in position $i$ if the content of the corresponding box equals $i$ modulo $\ell$. Then there is an arrow from vertex $u$ to vertex $v$ if the box of $v$ is above, or to the right of $u$, in the hook. Finally, the vertex $s_i$ corresponding to the hinge of the hook will be in position $0$. We mark this vertex. In this way, we get a Frobenius circle diagram. It is straight-forward to check that this defines a bijection between the set of all Frobenius circle diagrams and the set of all partitions.  \\[1ex]

The weight of a circle is simply the number of vertices in block zero (or the number of times the circle passes through zero). The \textit{weight} $\wt_{\ell}(C)$ of a Frobenius circle diagram is the weight of the longest circle. This notion is defined so that the weight of a Frobenius circle diagram equals the weight of the corresponding partition. 

 \section{The enhanced nilpotent cone}\label{sect:enc}

Let $V$ be an $n$-dimensional complex vector space, and $\N(V) \subset \End(V)$ the nilpotent cone. We denote by $\N(V)^{(x)}$ the closed subvariety $\{ \varphi \ | \ \varphi^x=0 \}$ of $x$-nilpotent endomorphisms. Each parabolic subgroup $P\subseteq \GL(V)$ acts by conjugation on $\N(V)^{(x)}$. This action has been studied in \cite{B2}. In particular, the main result of \textit{loc. cit.}, together with Theorem \ref{thm:enc_bijection} below, implies that: 
\begin{theorem}
 There are only finitely many $P$-orbits in $\N(V)^{(x)}$ if and only if 
 \begin{enumerate}
 \item $x\leq 2$,
 \item $P$ is maximal and $x=3$; or
 \item there exists a basis of $V$ for which $P$ has upper-block shape $(1,n-1)$ or $(n-1,1)$.
 \end{enumerate}
\end{theorem}
Cases 1. and 2. are described in detail in \cite{B2}. Let $\Q_{\infty}$ be the framed Jordan quiver: 
\begin{center}\small\begin{tikzpicture}
\matrix (m) [matrix of math nodes, row sep=0.01em,
column sep=1.5em, text height=0.5ex, text depth=0.1ex]
{ & \bullet & \bullet \\ & \infty & 0  \\ };
\path[->]
(m-1-2) edge node[above=0.05cm] {$v$} (m-1-3)
(m-1-3) edge [loop right] node{$\varphi$} (m-1-3);\end{tikzpicture}
\end{center} 
For each $x \ge 2$, define the algebra $\A_{\infty}(x) \coloneqq \C \Q_{\infty}/I_x$, where $I_x = (\varphi^x)$ is admissible. We fix the dimension vector $\mbf{d} = (d_{\infty},d_0)\coloneqq(1,n)$. The group $\GL_{\mbf{d}} = \GL_1 \times \GL(V)$ acts on $\Rep(\A_{\infty}(x),\mbf{d})$ and its orbits are the same as the orbits of $\GL(V)$. Therefore we consider $\Rep(\A_{\infty}(x),\mbf{d})$ as a $\GL(V)$-variety. As such, we have a canonical identification $\Rep(\A_{\infty}(x),\mbf{d}) = V \times \N(V)^{(x)}$, where $\GL(V)$ acts on the latter by $g.(v,N) = (g\cdot v,g\cdot N\cdot g^{-1})$ for all $g\in \GL(V), v\in V$ and $N\in \N(V)^{(x)}$. This setup is known as the \textit{enhanced nilpotent cone}.

 Let $\Rep(\A_{\infty}(x),\mbf{d})^{\circ}$ be the $\GL(V)$-stable open subset consisting of all representations where $v \neq 0$ and $V^{\circ} := V \smallsetminus \{ 0 \}$ so that $\Rep(\A_{\infty}(x),\mbf{d})^{\circ} = V^{\circ} \times \N(V)^{(x)}$. Let $P' \subset P$ be the subgroup where the entry in the $1 \times 1$ block is $1$.

\begin{theorem} \label{thm:enc_bijection}
There is an isomorphism of $\GL(V)$-varieties (resp. $\GL_{\mbf{d}}$-varieties)
\begin{enumerate}
\item[(1)] $V^{\circ} \times \N(V)^{(x)} \simeq \GL(V) \times^{P'} \N(V)^{(x)}$. 
\item[(2)] $\Rep(\A_{\infty}(x),\mbf{d})^{\circ} \simeq \GL_{\mbf{d}} \times^{P} \N(V)^{(x)}$.  
\end{enumerate}
\end{theorem}

\begin{proof}
Part (1). Choose $v \in V^{\circ}$ such that $P' = \Stab_{\GL(V)}(v)$. Since $V^{\circ} = \GL(V) \cdot v = \GL(V) / P'$, the isomorphism follows from Lemma \ref{thm:basis_transl} applied to the $\GL(V)$-equivariant projection map $V^{\circ} \times \N(V)^{(x)} \rightarrow V^{\circ}$. \\

Part (2). There is a well-defined group homomorphism $\eta : P \rightarrow \C^{\times}$ given by projection onto the $1 \times 1$ block. Embed $P \hookrightarrow \GL_{\mbf{d}} = \GL_1 \times \GL(V)$ by $p \mapsto (\eta(p^{-1}),p)$. Then, again, we choose $v \in  V^{\circ} \subset  \Rep(\A_{\infty}(x),\mbf{d})^{\circ}$ such that $\Stab_{\GL_{\mbf{d}}}(v) = P$. Since $V^{\circ} = \GL_{\mbf{d}} \cdot v \simeq \GL_{\mbf{d}}/ P$, the isomorphism again follows from Lemma \ref{thm:basis_transl} applied to the $\GL_{\mbf{d}}$-equivariant projection map $\Rep(\A_{\infty}(x),\mbf{d})^{\circ} \rightarrow V^{\circ}$.
\end{proof}
 Thus, there are bijections between the sets of orbits
\[
\left(V\times \N(V)^{(x)}\right)/\GL(V) ~~\xleftarrow{\alpha}~~\N^{(x)}/P'~~ \xrightarrow{\beta}~~ \Rep(\A_{\infty}(x),\mbf{d})^{\circ}/\GL_{\mbf{d}}.
\]
These bijections preserve orbit closure relations, dimensions of stabilizers (of single points) and codimension of orbits. Therefore the closure order relation, orbit dimensions and singularity type of the orbits in $\left(V\times \N(V)^{(x)}\right)/\GL(V)$, that were obtained in \cite{AH}, can be translated into the corresponding information for orbits in $\N^{(x)}/P'$. 

\subsection{Representation types}
We begin to examine the representation theory of the algebra $\A_{\infty}(x)$ by figuring out, if there are infinitely many representations of a fixed dimension vector before discussing the representation type of $\A_{\infty}(x)$. 
\begin{lemma}\label{lem:enc_dimv}
There are only finitely many isomorphism classes of $\A_{\infty}(x)$-representations of dimension vector $\mbf{d} = (d_{\infty},d_0)$ if and only if $d_{\infty}=1$ or $x\leq 3$. 
\end{lemma}
\begin{proof}
Finiteness for $d_{\infty}=1$ follows from \cite{AH}, finiteness for $x\leq 3$ follows from \cite{B2}. The fact that $\A_{\infty}(x)$ has infinite representation type in all other cases was shown in \cite{B2}, where explicit one parameter families were constructed.
\end{proof}
In order to decide whether the algebra $\A_{\infty}(x)$ is of finite representation type, tame or wild, we look at the universal covering quiver $\Gamma_{\infty}$ of $\Q_{\infty}$ \cite{Ga3}:

\[\begin{tikzpicture}
\matrix (m) [matrix of math nodes, row sep=0.65em,
column sep=1.5em, text height=1.1ex, text depth=0.2ex]
{\vdots & \vdots\\
 \bullet & \bullet\\
\bullet & \bullet\\
\bullet & \bullet\\
\vdots & \vdots \\};
\path[->]
(m-2-1) edge node[above=0.05cm] {$v^{(1)}$} (m-2-2)
(m-3-1) edge node[above=0.05cm] {$v^{(0)}$} (m-3-2)
(m-4-1) edge node[above=0.05cm] {$v^{(-1)}$} (m-4-2)
(m-2-2) edge node[right]{$\varphi^{(1)}$} (m-3-2)
(m-3-2) edge node[right]{$\varphi^{(0)}$} (m-4-2);\end{tikzpicture}\]
where $v^{(i)} : \varepsilon_{\infty}^{(i)} \rightarrow \varepsilon_{0}^{(i)}$ and $\varphi^{(i)} : \varepsilon_{0}^{(i)} \rightarrow  \varepsilon_{0}^{(i-1)}$. Together with the admissible ideal $(\varphi^x)$ (this notation means that the ideal is generated by every vertical path of length $x$), we obtain the covering algebra $\Gamma_{\infty}(x):= \C \Gamma_{\infty} /(\varphi^x)$. If $\Gamma_{\infty}(x)$ is of wild representation type, then via the covering functor \cite{Ga3}, the algebra $\A_{\infty}(x)$ is of wild representation type, as well. 
\begin{lemma}\label{lem:enc_reptype}
The algebra $\A_{\infty}(x)$ is of finite representation type if and only if $x\leq 3$, and of wild representation type otherwise.
\end{lemma}

\begin{proof}
Finiteness for $x\leq 3$ follows from Lemma \ref{lem:enc_dimv}. The algebra $\Gamma(x)$ is strongly simply connected, since every convex subcategory is triangular and fulfills the separation condition: the radicals of all projective indecomposables are indecomposable. Thus, Corollary \ref{lem:wild_crit} implies that $\Gamma_{\infty}(x)$ has wild representation type if and only if there is a dimension vector $\mbf{d}\in\mathbf{N} (\Gamma_{\infty})_0$, such that $q_{\Gamma_{\infty}(x)}(\mbf{d})\leq -1$. If $x\geq 4$, one such dimension vector is: 
  \begin{center}
  \begin{tikzpicture}
\matrix (m) [matrix of math nodes, row sep=0.1em,
column sep=-0.5em, text height=0.1ex, text depth=0.1ex]
{ 1 &  2\\
2 &  3\\
2  & 3\\
1 &  2\\
 };
\end{tikzpicture}  
  \end{center}
Hence $\A_{\infty}(x)$ is wild, too. 
\end{proof}
 
\subsection{The indecomposable representations}\label{ssect:enc_indec}

In this section, we classify all indecomposable representations $M$ of $\A_{\infty}(x)$ that have the property that $(\dim M)_{\infty} \le 1$. 

\subsubsection{Classification of indecomposables of dimension vector $(0,n)$}\label{sssect:enc_indecs0}

The Jordan normal form implies that there is (up to isomorphism) exactly one indecomposable representation with dimension vector $(0,n)$, which is given by the Jordan block of size $n$. We denote the natural indecomposable representative by
\[\begin{tikzpicture}
\matrix (m) [matrix of math nodes, row sep=0.05em,
column sep=2em, text height=1.5ex, text depth=0.2ex]
{U(0,n) : &0 & \C^{n}\\ };
\path[->]
(m-1-2) edge node[above=0.05cm] {} (m-1-3)
(m-1-3) edge [loop right] node{$J_n$} (m-1-3);\end{tikzpicture}\]

\subsubsection{Classification of indecomposables of dimension vector $(1,n)$}\label{sssect:enc_indecs1}

Some additional work is required to understand the indecomposable representations with dimension vector $(1,n)$. We recall some notions from \cite{AH}. First, given a nilpotent matrix $X$ of type $\lambda \vdash n$, a \textit{Jordan basis} $\{ v_{i,j} \ | \ 1 \le i \le \ell(\lambda), \ 1 \le j \le \lambda_i \}$ is a basis of $V$ such that $X(v_{i,j}) = v_{i,j-1}$ if $j > 1$ and $0$ otherwise. Similarly, if $(v,X)$ is a representation of $\mc{A}_{\infty}(n)$ of dimension $(1,n)$, then a \textit{normal basis} for $(v,X)$ is a Jordan basis such that 
$$
v = \sum_{i = 1}^{\ell(\mu)} v_{i,\mu_i},
$$
where $\mu \subset \lambda$ is a partition such that $\nu_i = \lambda_i - \mu_i$ also defines a partition. Thus, given a bipartition $(\mu;\nu)$ of $n$, by choosing a normal basis, one gets an element $(v,X) \in \N_{\infty}(n)$. By \cite[Proposition 2.3]{AH}, the orbit of $(v,X)$ does not depend on the choice of normal basis and the rule $(\mu;\nu) \mapsto G \cdot (v,X) =: \Xi(\mu;\nu)$ is a bijection $\Xi : \mc{P}_2(n) \rightarrow \N_{\infty}(n)/G$. 
 
\begin{lemma}\label{lem:enc_indec1n}
 There is a bijection $\Upsilon'$ from $\Pa_F(n)$ to the set of indecomposable representations (up to isomorphism) of $\A_{\infty}(n)$ with dimension vector $(1,n)$. Given a Frobenius partition $(\mbf{a},\mbf{b})$ of length $k$, the latter are represented by
 \[\begin{tikzpicture}
\matrix (m) [matrix of math nodes, row sep=0.05em,
column sep=2em, text height=1.5ex, text depth=0.2ex]
{\Upsilon'(\mbf{a},\mbf{b}) : & \C & \C^{m}\\ };
\path[->]
(m-1-2) edge node[above=0.05cm] {v} (m-1-3)
(m-1-3) edge [loop right] node{$J_{P(\mbf{a},\mbf{b})}$} (m-1-3);\end{tikzpicture} \]
where $v=\sum_{i=1}^{k} v_{i,a_i +1}$, with respect to a Jordan basis $\{ v_{i,j} \}$ for $J_{P(\mbf{a},\mbf{b})}$.
\end{lemma}

\begin{remark}
Lemma \ref{lem:enc_indec1n} also appeared in the recent preprint \cite{DoGinT}, as Lemma 11.2.1.
\end{remark}

\begin{example}
Consider the Frobenius partition $(\mbf{a},\mbf{b})=((1,0),(3,0))$. Then $P(\mbf{a},\mbf{b})=(5,1)$. Thus, we can find a basis $\{ e_ i\}$ of $\C^6$, such that

\[\begin{tikzpicture}
\matrix (m) [matrix of math nodes, row sep=0.05em,
column sep=2em, text height=1.5ex, text depth=0.2ex]
{\Upsilon'(\mbf{a},\mbf{b}) : & \C & & \C^{6}\\ };
\path[->]
(m-1-2) edge node[above=0.05cm] {$e_2 + e_6$} (m-1-4)
(m-1-4) edge [loop right] node{$J_{(5,1)}$} (m-1-4);\end{tikzpicture} \]


\end{example}

\begin{proof}[Proof of Lemma \ref{lem:enc_indec1n}]
Given a Frobenius partition $(\mbf{a},\mbf{b})$ of length $k$, let $\mu = (a_1 + 1, \ds,a_k +1)$ and $\nu = (b_1, \ds, b_k)$ so that $(\mu,\nu)$ is a bipartition of $n$ with $\mu + \nu = P(\mbf{a},\mbf{b})$. This identifies $\mc{P}_F(n)$ with the subset 
$$
\mc{P}_{2,F}(n) := \{ (\mu,\nu) \in \mc{P}_2(n) \ |  \ \mu_1 > \cdots > \mu_k > 0, \textrm{ and } \nu_1 > \cdots > \nu_k \ge 0 \}.
$$ 
Then we need to show that a) if $(v,X) \in \mc{O}_{(\mu,\nu)}$ with $(\mu,\nu) \notin \mc{P}_{2,F}(n)$, then $(v,X)$ is decomposable; and b) if $(\mu,\nu) \in \mc{P}_{2,F}(n)$, then $(v,X)$ is indecomposable. 
\begin{enumerate}
\item[(a)] We assume that $(\mu,\nu) \notin \mc{P}_{2,F}(n)$. Let $\lambda = \mu + \nu$. If $\ell(\lambda) = 1$, then $(\mu;\nu) = ((n-k),(k))$ for some $k \le n$. These all belong to $\mc{P}_{2,F}(n)$, except when $k = n$. This corresponds to $X$ a single Jordan block and $v = 0$, which is clearly decomposable. Also, we note that if $\mu_k = 0$ and $\nu_k \neq 0$, then let $V_1$ be the span of $\{ v_{i,j} \ |  \ i < k \}$ and $V_2$ the span of the $\{ v_{k,j} \}$. Then $v \in V_1$ and $X(V_i) \subset V_i$ (with $X |_{V_2}$ a Jordan block of length $\nu_k$). Thus, $(v,X)$ is decomposable. Therefore, we may assume that $k > 1$ and $\mu_k \neq 0$.

There exists some $i < k$ such that $\mu_i = \mu_{i+1}$ or $\nu_{i} = \nu_{i+1}$. We begin by assuming that $\mu_i = \mu_{i+1}$. It is enough to assume that $i = 1$ and $\lambda = (\lambda_1, \lambda_2)$. Let $v_{1,1}, \ds, v_{1,\lambda_1}, v_{2,1}, \ds, v_{2,\lambda_{2}}$ be a normal basis of $V$ with $v = v_{1,\mu_1} + v_{2,\mu_2}$. There are two subcases:
\begin{itemize}
\item[(i)] $\lambda_{1} > \lambda_{2}$. We take new basis $v_{1,\lambda_1}' = v_{1,\lambda_1} + v_{2,\lambda_2}, v_{1,\lambda_1 -1}' = v_{1,\lambda_1 -1} + v_{2,\lambda_2 -1}, \ds$ and $v_{2,\lambda_2}' = v_{1,\lambda_2} + v_{2,\lambda_2}, v_{2,\lambda_2 -1}' = v_{1,\lambda_2 -1} + v_{2,\lambda_2 -1}, \ds$. Then $v$ belongs to the subspace spanned by the $v_{2,j}'$ and the representation is decomposable. We note that with respect to the new basis $(X,v)$ has type $\mu' = (0,\mu_2), \nu' = (\mu_1  + \nu_1,\nu_{2})$, which is \textit{not} a normal form in the sense of \cite{AH}. 

\item[(ii)] $\lambda_1 = \lambda_2$. Take a new basis $v_{1,\lambda_1}' = v_{1,\lambda_1} - v_{2,\lambda_1}, v_{1,\lambda_1 -1}' = v_{1,\lambda_1 -1} - v_{2,\lambda_1 -1}, \ds$ and $v_{2,\lambda_2}' = v_{1,\lambda_1} + v_{2,\lambda_1}, v_{2,\lambda_2 -1}' = v_{1,\lambda_1 -1} + v_{2,\lambda_1 -1}, \ds$. Again $v$ is in the subspace spanned by the $v_{2,j}'$ and $(X,v)$ has type $\mu' = (0,\mu_2), \nu' = (\mu_1 + \nu_1,\nu_{2})$, which is \textit{not} a normal form. 
\end{itemize}

When $\nu_{i} = \nu_{i+1}$, we can take $i = 1$ again. Let $v_{1,1}, \ds, v_{1,\lambda_1}, v_{2,1}, \ds, v_{2,\lambda_{2}}$ be a normal basis of $V$ with 
$$
v = v_{1,\mu_1} + v_{2,\mu_2} = v_{1,\lambda_1 - \nu_1} + v_{2,\lambda_2 - \nu_2}.
$$
Again, one considers the two subcases (i) $\lambda_{1} > \lambda_{2}$ and (ii) $\lambda_1 = \lambda_2$. Repeating the above argument shows that these representations are decomposable. We note that, when one takes the new basis as above, $(v,X)$ has type $\mu' = (\mu_1,0), \nu' = (\nu_1,\mu_2 + \nu_{2})$, resp. type $\mu' = (\mu_1,0), \nu' = (\nu_1,\mu_2 + \nu_{2})$, when  $\lambda_{1} > \lambda_{2}$, resp. when $\lambda_1 = \lambda_2$. 

\item[(b)] Take $(\mu;\nu) \in \mc{P}_{2,F}(n)$, and assume that the corresponding representation is decomposable i.e. $V = V_1 \oplus V_2$ with $v \in V_1$ and $X(V_i) \subset V_i$. By \cite[Corollary 2.9]{AH}, the Jordan type of $X |_{Z_{\g}(X) \cdot v}$ is $\mu$ and the type of $X |_{V / Z_{\g}(X) \cdot v}$ is $\nu$. If the Jordan type of $X |_{V_1}$ is $\eta$ and the type of $X |_{V_2}$ is $\zeta$, then the fact that $\g_X \cdot v \subset V_1$ implies that $\mu \subseteq \eta$ and $\zeta \subseteq \nu$. Here $\lambda = \eta \sqcup \zeta$. The fact that $V_2 \neq 0$ implies that there exists some $i$ such that $\mu_i = 0$ but $\nu_i \neq 0$. But this contradicts the fact that $(\mu;\nu) \in \mc{P}_{2,F}(n)$. 
\end{enumerate}
\end{proof}

The bijection $\varphi: \Pa(n)\rightarrow \Pa_F(n)$ from Subsection \ref{ssect:part_YD} immediately yields the following corollary.
\begin{corollary}
There is a bijection $\Upsilon$ between $\Pa(n)$ and the set of indecomposable representations (up to isomorphism) of dimension vector $(1,n)$, given by $\Upsilon\coloneqq\Upsilon'\circ \varphi$.
\end{corollary}
\begin{example}
The indecomposable representations of dimension vector $(1,6)$ are (up to isomorphism) given by
$$\begin{tabular}{|l|l|l|l|}\hline
Jordan blocks & Representation & Frobenius partition & Partition\\ \hline
(6) & $(e_1,J_6)$ & $((0),(5))$ & $(6)$\\ \hline
(6) & $(e_2,J_6)$ & $((1),(4))$ & $(5,1)$\\ \hline
(6) & $(e_3,J_6)$ & $((2),(3))$ & $(4,1,1)$\\ \hline
(6) & $(e_4,J_6)$ & $((3),(2))$ & $(3,1,1,1)$\\ \hline
(6) & $(e_5,J_6)$ & $((4),(1))$ & $(2,1,1,1,1)$\\ \hline
(6) & $(e_6,J_6)$ & $((5),(0))$ & $(1,1,1,1,1,1)$\\ \hline
(5,1) & $(e_2+e_6,J_{5,1})$ & $((1,0),(3,0))$ & $(4,2)$ \\ \hline
(5,1) & $(e_3+e_6,J_{5,1})$ & $((2,0),(2,0))$ & $(3,2,1)$ \\ \hline
(5,1) & $(e_4+e_6,J_{5,1})$ & $((3,0),(1,0))$ & $(2,2,1,1)$ \\ \hline
(4,2) & $(e_2+e_5,J_{4,2})$ & $((1,0),(2,1))$ & $(3,3)$ \\ \hline
(4,2) & $(e_3+e_6,J_{4,2})$ & $((2,1),(1,0))$ & $(2,2,2)$ \\\hline
\end{tabular}$$
\end{example}
Given a partition $\lambda$, we write $M_{\lambda} := \Upsilon(\lambda)$. Note that the Jordan block sizes of $M_\lambda$ are not given by $\lambda$, but by $P(\varphi({\lambda}))$. We deduce that:

\begin{theorem}\label{thm:enc_class}
Every representation $M$ in $\N_{\infty}(n) := \Rep(\A_{\infty}(n),(1,n))$ decomposes as a direct sum 
\begin{equation}\label{eq:Mbipart}
M\simeq M_{\lambda}\oplus\bigoplus_{i=1}^{\ell(\mu)} U(0,\mu_i)
\end{equation}
for some bipartition $(\lambda;\mu) \in \mc{P}_2(n)$. 
\end{theorem}

We denote the $G$-orbit of the representation (\ref{eq:Mbipart}) by $\mc{O}_{(\lambda;\mu)}$. 

\begin{corollary}
There is a bijection $\Phi : (\lambda;\mu) \mapsto \mc{O}_{(\lambda;\mu)}$ from $\Pa_2(n)$ to the orbits in the enhanced nilpotent cone.
\end{corollary}

\subsection{Translating between the different parameterizations}\label{ssect:enc_transl}

The $\GL(V)$-orbits in $V \times \N(V)$ were first studied by Bernstein \cite{Ber}. It particular, it was noted there that there are only finitely many orbits. Explicit representatives of these orbits were independently given by Achar-Henderson \cite[Proposition 2.3]{AH} and Travkin \cite[Theorem 1]{Tr}. Recall that the parameterization in \textit{loc. cit.} is given by $\Xi : \mc{P}_2(n) \rightarrow \N_{\infty}(n) / G$. \\

We define $\Psi : \mc{P}_2(n) \rightarrow \mc{P}_2(n)$ as follows. Take $(\mu;\nu) \in \mc{P}_2(n)$ and let $k = \ell(\lambda)$, where $\lambda = \mu + \nu$. Then $i$ belongs to the set $\mathrm{Re}(\mu;\nu) \subset \{ 1, \ds, k \}$, the set of removable rows, if and only if one of the following holds: 
\begin{itemize}
\item[(a)] $\mu_i = \mu_{i+1}$,
\item[(b)] $\nu_{i-1} = \nu_i$; or
\item[(c)] $i = k $ and $\mu_k = 0$. 
\end{itemize}
If $\mathrm{Re}(\mu;\nu) = \{ i_1, \ds, i_r \}$ then $\zeta := (\lambda_{i_1} \ge \cdots \ge \lambda_{i_r})$. Let $\mu' = (\mu_1 - 1, \ds, \mu_k - 1)$ and
\begin{align*}
\widehat{\mu} & = \mu' \textrm{ with $\mu_{i_1}', \ds, \mu_{i_r}'$ removed,} \\
\widehat{\nu} & = \nu \textrm{ with $\nu_{i_1}, \ds, \nu_{i_r}$ removed.}
\end{align*}
Then $(\widehat{\mu},\widehat{\nu}) = \varphi(\eta)$ is a Frobenius partition, and 
$$
\Psi(\mu;\nu) := (\eta;\zeta) \in \mc{P}_2(n). 
$$

\begin{example}
If $(\mu;\nu) = ((4,4,3,1),(3,2,2))$, then $(\eta,\zeta) = ((3,1,1,1),(7,5))$. 
\end{example}

\begin{theorem}\label{thm:travkin}
The map $\Psi$ is a bijection such that 
$$
\Phi \circ \Psi (\mu;\nu) = \mc{O}_{\Psi(\mu;\nu)} = \Xi (\mu;\nu). 
$$
\end{theorem}

\begin{proof}
We check that $\Phi \circ \Psi(\mu;\nu) = \Xi(\mu;\nu)$ for all $(\mu;\nu)$. This will imply that $\Psi$ is a bijection. This is essentially already contained in the proof of Lemma \ref{lem:enc_indec1n}. Let $M = (v,X) \in \Xi(\mu;\nu)$. This means that there is a normal basis $\{ v_{i,j} \}$ of $V$ such that $X$ has Jordan type $\lambda = \mu + \nu$ and $v = \sum_{i = 1}^{\ell(\mu)} v_{i,\mu_i}$. We wish to show that 
$$
M\simeq M' \oplus\bigoplus_{i=1}^{\ell(\zeta)} U(0,\zeta_i)
$$
where $M'$ is isomorphic to the representation $\Upsilon'(\widehat{\mu},\widehat{\nu})$ of Lemma \ref{lem:enc_indec1n}. The proof is by induction on $n = |\lambda|$. Assume that there exists $i \in \{1 ,\ds, k \}$ such that $\mu_i = \mu_{i+1}$. Then, the proof of  Lemma \ref{lem:enc_indec1n} shows that $M \simeq M_1 \oplus U(0,\lambda_i)$, where $M_1$ belongs to $\Xi(\mu'';\nu'')$, where $\mu''$ is $\mu$ with $\mu_i$ removed and $\nu''$ is $\nu$ with $\nu_i$ removed. By induction, $\Xi(\mu'';\nu'') =  \Phi \circ \Psi (\mu'';\nu'')$ and hence $\Xi(\mu;\nu) =  \Phi \circ \Psi (\mu;\nu)$. In exactly the same way, if $\nu_{i-1} = \nu_i$ or if $i = k$ and $\mu_k = 0$ then $M$ has a summand isomorphic to $U(0,\lambda_i)$. This reduces us to the situation where  $\mu_1 > \cdots > \mu_k > 0$ and $\nu_1 > \cdots > \nu_k \ge 0$ i.e. we may assume that $(\mu;\nu)$ belongs to the set $\mc{P}_{2,F}(n)$ defined in the proof of  Lemma \ref{lem:enc_indec1n}. If we set $  \widehat{\mu} = (\mu_1 - 1, \ds, \mu_k - 1)$ and $\widehat{\nu} = \nu$, then Lemma \ref{lem:enc_indec1n} says that  $(\widehat{\mu},\widehat{\nu})$ is a Frobenius partition and $M \in \Upsilon'(\widehat{\mu},\widehat{\nu})$. This completes the proof of the theorem. 
\end{proof}

\begin{example}
We consider the case $n = 3$. Then
\begin{displaymath}
\begin{array}{c|c}
(\mu;\nu) & \Psi(\mu;\nu) \\
\hline
((3);\emptyset)  & ((1,1,1);\emptyset) \\
((2,1);\emptyset)  & ((1,1);(1)) \\
((1,1,1);\emptyset)  & ((1);(1,1)) \\
((2);(1))  & ((2,1);\emptyset) \\
((1,1);(1))  & ((1);(2)) \\
((1);(2))  & ((3);\emptyset) \\
((1);(1,1))  & ((2);(1)) \\
(\emptyset;(3))  & (\emptyset;(3)) \\
(\emptyset;(2,1))  & (\emptyset;(2,1)) \\
(\emptyset;(1,1,1))  & (\emptyset;(1,1,1)) 
\end{array}
\end{displaymath}
\end{example}

We have shown  that there is an explicit bijection from the set of bipartitions of $n$ to itself, that intertwines our parameterization of $G$-orbits on the enhanced nilpotent cone with the parameterization given in \cite{AH} and \cite{Tr}. This bijection is very non-trivial and we hope to figure out a better combinatorial understanding of its properties in the near future.

\section{The enhanced cyclic nilpotent cone}\label{sec:cyclicenh}

The results described above for the enhanced nilpotent cone all have analogues for the enhanced cyclic nilpotent cone. As one might expect, this situation is combinatorially more involved, but the approach is similar. \\[1ex]
Let $\Q_{\infty}(\ell)$ be the enhanced cyclic quiver with $\ell+1$ vertices.
\begin{center}
\scalebox{0.6}{
\begin{tikzpicture}[->,>=stealth',shorten >=1pt,auto,node distance=3cm,
  thick,  node/.style={}]
\node (99) {$\bullet_{\infty}$};
\node (1)[right of=99] {$\bullet_0$};
\node (2) [above right=1 and 0.5 of 1] {$\bullet_1$};
\node (7) [below right=1 and 0.5 of 1] {$\bullet_{\ell-1}$};
\node (3) [above right=0.25 and 1 of 2] {$\bullet_2$};
\node (6) [below right=0.25 and 1 of 7] {$\bullet_{\ell-2}$};
\node (4) [below right=0.25 and 1 of 3] {$\bullet_3$};
\node (5) [above right=0.25 and 1 of 6] {$\bullet_{\ell-3}$};
\path[->]
(99) edge node {$v$} (1)
(1) edge [bend left=15] node {$\varphi_0$} (2)
(2) edge [bend left=15] node {$\varphi_1$} (3)
(3) edge [bend left=15] node {$\varphi_2$} (4)
(5) edge [bend left=15] node {$\varphi_{\ell-3}$} (6)
(6) edge [bend left=15] node {$\varphi_{\ell-2}$} (7)
(7) edge [bend left=15] node {$\varphi_{\ell-1}$} (1);
\path[-,dotted]
(4) edge [bend left=35] (5);
\end{tikzpicture}}
\end{center}
We define the cyclic enhanced algebra to be 
$$
\A_{\infty}(\ell,x):=\C \Q_{\infty}(\ell)/\langle(\varphi_{\ell-1}\circ \cdots \circ\varphi_0)^x\rangle.
$$
Let us fix a dimension vector $\mbf{d}_{\infty} :=(1,d_0,...,d_{\ell-1})$ of $\Q_{\infty}(\ell)$. The group $\GL_{\mbf{d}_{\infty}}:=\GL_1\times \prod_{i=0}^{\ell-1}\GL_{d_i}$ acts on the representation variety $\Rep(\A_{\infty}(\ell,x),\mbf{d}_{\infty})$. Recall that $\Q(\ell)$ denotes the unframed cyclic quiver. We define $\A(\ell,x):=\C \Q(\ell)/\langle(\varphi_{\ell-1}\circ...\circ\varphi_0)^x\rangle$. This algebra has been studied previously by Kempken \cite{Ke}. Let us fix the dimension vector $\mbf{d} :=(d_0,...,d_{\ell-1})$. The group $\GL_{\mbf{d}}:=\prod_{i=0}^{\ell-1}\GL_{d_i}$ acts on the representation variety $\Rep(\A(\ell,x),\mbf{d})$.

Just as in section \ref{sect:enc}, one can relate orbits in the cyclic nilpotent cone $\Rep(\A(\ell,x),\mbf{d})$ and in the enhanced cyclic nilpotent cone $\Rep(\A_{\infty}(\ell,x),\mbf{d}_{\infty})$. If $V = \C^{d_0}$, identified in the obvious way with a subspace of $\Rep(\A_{\infty}(\ell,x),\mbf{d}_{\infty})$, then we take $V^{\circ} = V \smallsetminus \{ 0 \}$ and let $\Rep(\A_{\infty}(\ell,x),\mbf{d}_{\infty})^{\circ}$ denote its preimage under the projection $\Rep(\A_{\infty}(\ell,x),\mbf{d}_{\infty}) \rightarrow V$. Choose $v \in V^{\circ} \subset \Rep(\A_{\infty}(\ell,x),\mbf{d}_{\infty})$ and let 
$$
P = \Stab_{\GL_{\mbf{d}_{\infty}}}(v), \quad P' = \Stab_{\GL_{\mbf{d}}}(v). 
$$
Analogous to Theorem \ref{thm:enc_bijection}, we have 

\begin{theorem} \label{thm:cenc_bijection}
There is an isomorphism of $\GL_{\mbf{d}}$-varieties (resp. of $\GL_{\mbf{d}_{\infty}}$-varieties):
\begin{enumerate}
\item  $\Rep(\A_{\infty}(\ell,x),\mbf{d}_{\infty})^{\circ} \cong \GL_{\mbf{d}} \times^{P'} \Rep(\A(\ell,x),\mbf{d})$.
\item  $\Rep(\A_{\infty}(\ell,x),\mbf{d}_{\infty})^{\circ} \cong \GL_{\mbf{d}_{\infty}} \times^{P} \Rep(\A(\ell,x),\mbf{d})$.
\end{enumerate}
\end{theorem}

\begin{remark}\label{rem:theta}
Let $N = d_0 + \cdots + d_{\ell-1}$. There is an automorphism $\theta$ of $\g := \mf{gl}_N$ such that 
$$
\g_1 := \left\{ X \in \g  \ \Big| \ \theta(X) = \exp \left( \frac{2 \pi \sqrt{-1}}{\ell} \right) \right\}
$$
is canonically identified with $\Rep(\Q(\ell),\mbf{d})$. The space $\g_1$ is a representation of $\GL_N^{\theta} = \GL_{\mbf{d}}$, and is an example of a $\theta$-representation as introduced and studied by Vinberg \cite{Vinberg}. Under the above identification, the cyclic nilpotent cone $\Rep(\A(\ell,x),\mbf{d})$ is precisely the nilcone in the $\theta$-representation $\g_1$. Therefore one can view Theorem \ref{thm:cenc_bijection} as a first step in a programme to study parabolic conjugacy classes in the nilcone of $\theta$-representations. In particular, it raises the following problem:\\

Classify all triples $(G,\theta,P)$, where $G$ is a reductive group over $K$, $\theta$ is a finite automorphism of $\g = \mathrm{Lie} \ G$  and $P \subset G^{\theta}$ is a parabolic subgroup such that the number of $P$-orbits in the nilcone $\N(\g_1)$ is finite. 
\end{remark}

\subsection{Representation types}\label{sec:reptype}

We begin by classifying the representation type of the algebra $\A_{\infty}(\ell,x)$. The universal covering quiver $\Gamma_{\infty}(\ell)$ is given by
\[\begin{tikzpicture}
\matrix (m) [matrix of math nodes, row sep=0.95em,
column sep=1.5em, text height=1.5ex, text depth=0.2ex]
{\vdots & \vdots\\
 \bullet & \bullet\\
\bullet & \bullet\\
\bullet & \bullet\\
\vdots & \vdots \\};
\path[->]
(m-2-1) edge node[above] {$v^{(1)}$} (m-2-2)
(m-3-1) edge node[above] {$v^{(0)}$} (m-3-2)
(m-4-1) edge node[above] {$v^{(-1)}$} (m-4-2)
(m-2-2) edge[dashed] node[right]{$\underline{\varphi}^{(i)}$} (m-3-2)
(m-3-2) edge[dashed] node[right]{$\underline{\varphi}^{(i-1)}$} (m-4-2);\end{tikzpicture}\]
Here, $\underline{\varphi}^{(i)}= \varphi_{\ell-1}^{(i)} \circ \cdots \circ \varphi_0^{(i)}$ is a path of length $\ell$. The quotient of this path algebra by the relations $\langle \underline{\varphi}^{(i)} \circ \cdots \circ \underline{\varphi}^{(i+x-1)} \circ \underline{\varphi}^{(i+x)} \ | \ i \in \Z \rangle$ gives the covering algebra $\Gamma_{\infty}(\ell,x):=\C \Gamma_{\infty}(\ell)/(\underline{\varphi}^x)$. If $\Gamma_{\infty}(\ell,x)$ is of wild representation type, then via the covering functor \cite{Ga3}, the algebra $\A_{\infty}(\ell,x)$ is of wild representation type as well.  Since the covering algebra is strongly simply connected (as every projective indecomposable admits at every vertex a vector space of dimension at most $1$), we can make use of the results of subsection \ref{ssect:repTypes}; in particular of Lemma \ref{lem:wild_crit}.

Furthermore, since $\Gamma_{\infty}(\ell,x)$ is locally bounded (our ideal cancels infinite paths) and $\mathbf{Z}$ acts freely by shifts, we know by \cite{Ga3}: If $\Gamma_{\infty}(\ell,x)$ is locally of finite representation type, then $\A_{\infty}(\ell,x)$ is of finite representation type and every indecomposable representation is obtained from an indecomposable $\Gamma_{\infty}(\ell,x)$-representation (via the obvious functor which builds direct sums of vector spaces in the same "column" and linear maps accordingly).

\begin{lemma}\label{lem:cenc_reptype} 
The algebra $\A_{\infty}(\ell,x)$ has
\begin{enumerate}
\item[(a)]   finite representation type iff $(\ell,x)\in\{(1,1),(1,2),(1,3),(2,1),(3,1)\}$,
\item[(b)] tame representation type iff $(\ell,x)\in\{(2,2),(4,1)\}$.

\end{enumerate}
In every remaining case, the algebra $\A_{\infty}(\ell,x)$ is of wild representation type.
\end{lemma}

\begin{proof}
This is a case by case analysis. 
\begin{itemize}
\item Firstly, let $\ell=1$, that is, we are in the situation of the enhanced nilpotent cone. Then every case follows from Lemma \ref{lem:enc_reptype}.

\item Let $\ell=2$. 
\begin{itemize}
\item For $x=1$, by knitting, we compute the Auslander-Reiten quiver of $\Gamma_{\infty}(2,1)$, which is finite.  It is cyclic by means of a shift of the $\mathbf{Z}$-action and is depicted in Appendix \ref{ssect:arq21}. There, given a representation $M$ of a finite slice of $\Gamma_{\infty}(2,1)$, we denote by $M^{(i)}$ for $i\in\mathbf{Z}$ the shifted representation $M$, such that the support of $M^{(i)}$ is non-zero in the $i$-th row (numbered from bottom to top) of $\Gamma_{\infty}(2,1)$, but zero below.
By Covering Theory \cite{Ga3}, the algebra $\A_{\infty}(2,1)$ is, thus, of finite representation type. 
\item If $x=2$, then the algebra is tame: The covering algebra contains an Euclidean subquiver of type $\widetilde{\mathsf{D}}_6$ and hence $\A_{\infty}(2,2)$ has infinite representation type. It is indeed tame by \cite[Theorem 2.4]{Sko3}: The Galois covering is strongly simply connected locally bounded and the algebra does not contain a convex subcategory which is
hypercritical (see a list in \cite{Un}) or pg-critical (see a list in \cite{NoeSk}).
\item For $x\geq 3$, there is a dimension vector with negative quadratic form, and wildness follows from Lemma \ref{lem:wild_crit}:
  \begin{center}
  \begin{tikzpicture}
\matrix (m) [matrix of math nodes, row sep=0.1em,
column sep=-0.5em, text height=0.1ex, text depth=0.1ex]
{  & 1\\
1 &  2\\
&  2\\
1 &  3\\
 &  2\\
1 &  2\\
 &  1\\
 };
\end{tikzpicture}  
  \end{center}
\end{itemize}

\item Let $\ell=3$. 
\begin{itemize}
\item If $x=1$, as above, we compute the Auslander-Reiten quiver of $\Gamma_{\infty}(3,1)$, which is finite.  It is also cyclic by means of a shift of the $\mathbf{Z}$-action and is depicted in Appendix \ref{ssect:arq31}. There, given a representation $M$ of a finite slice of $\Gamma_{\infty}(3,1)$, we denote by $M^{(i)}$ for $i\in\mathbf{Z}$ the shifted representation $M$, such that the support of $M^{(i)}$ is non-zero in the $i$-th row (numbered from bottom to top) of $\Gamma_{\infty}(3,1)$, but zero below.
By Covering Theory \cite{Ga3}, the algebra $\A_{\infty}(3,1)$ is, thus, of finite representation type. 
\item Let $x\geq 2$. The following dimension vector, for which $q_{\A}(\mbf{d}) = -1$, proves wildness of $\A_{\infty}(3,2)$ by Lemma \ref{lem:wild_crit} and induces a $2$-parameter family of non-isomorphic representations:
  \begin{center}
  \begin{tikzpicture}
\matrix (m) [matrix of math nodes, row sep=0.1em,
column sep=-0.5em, text height=0.1ex, text depth=0.1ex]
{  & 1\\
 & 2\\
1 &  3\\
 &  3\\
  & 3\\
1 &  3\\
 &  2\\
  & 1\\
 };
\end{tikzpicture}  
  \end{center}
Wildness of   $\A_{\infty}(3,x)$ for $x\geq 2$ follows.

\end{itemize}

\item Let $\ell = 4$. The covering quiver contains an euclidean (and therefore tame) subquiver of type $\widetilde{\mathsf{E}}_7$, thus, we always have infinite representation type.
\begin{itemize}
\item If $x=1$, then  then algebra $\A_{\infty}(4,1)$ is tame by \cite[Theorem 2.4]{Sko3}: The Galois covering is strongly simply connected locally bounded and the algebra does not contain a convex subcategory which is
hypercritical (see a list in \cite{Un}) or pg-critical (see a list in \cite{NoeSk}). 

\item For $x\geq 2$, the algebra $\A_{\infty}(4,x)$ is of wild representation type, since the covering quiver contains the wild subquiver (see \cite{Un}):

\[\begin{tikzpicture}
\matrix (m) [matrix of math nodes, row sep=0.95em,
column sep=1.5em, text height=1.5ex, text depth=0.2ex]
{ &\bullet &&&&\bullet&&\\
 \bullet & \bullet& \bullet &\bullet& \bullet& \bullet & \bullet & \bullet\\};
\path[-]
(m-2-1) edge  (m-2-2)
(m-2-2) edge (m-2-3)
(m-2-3) edge  (m-2-4)
(m-2-4) edge (m-2-5)
(m-2-5) edge (m-2-6)
(m-2-6) edge (m-2-7)
(m-2-7) edge (m-2-8)
(m-1-2) edge (m-2-2)
(m-1-6) edge (m-2-6);\end{tikzpicture}\]

\end{itemize} 

\item If $\ell \ge 5$, then the algebra $\A_{\infty}(\ell,x)$ has wild representation type. In this case, the covering quiver  $\Gamma_{\infty}(\ell,x)$ contains the wild subquiver

\[\begin{tikzpicture}
\matrix (m) [matrix of math nodes, row sep=0.95em,
column sep=1.5em, text height=1.5ex, text depth=0.2ex]
{ & &&&\bullet&&&&\\
 \bullet & \bullet& \bullet & \bullet& \bullet & \bullet& \bullet& \bullet & \bullet\\};
\path[-]
(m-2-1) edge  (m-2-2)
(m-2-2) edge (m-2-3)
(m-2-3) edge  (m-2-4)
(m-2-4) edge (m-2-5)
(m-2-5) edge (m-2-6)
(m-2-6) edge  (m-2-7)
(m-2-7) edge (m-2-8)
(m-2-8) edge (m-2-9)
(m-1-5) edge (m-2-5);\end{tikzpicture}\]
Thus, $\A_{\infty}(5,x)$ is of wild representation type for all $x$.

\end{itemize}
\end{proof}

\subsection{The indecomposable representations}\label{ssect:ecnc_indec11}
As for the usual enhanced nilpotent cone, we consider separately the two cases:  indecomposable representations $M$ of $\A_{\infty}(\ell,x)$ with $(\dim M)_{\infty} = 1$; or indecomposable representations $M$ with $(\dim M)_{\infty} = 0$.

\subsubsection{Classification of indecomposables of dimension vector $(0,*)$}\label{sssect:cenc_indecs0}

In this case, we are basically studying indecomposable representations of $\A(\ell,x)$. For $i\in\Z_{\ell}$ and $N\in\mathbf{N}\backslash \{0\}$, let $U(i,N)$ be the $N$-dimensional indecomposable module defined as follows:
as a $\C$-vector space, it has a basis $v_0,\dots,v_{N-1}$, with $v_k$ being a basis vector of the vector space at vertex $i+k\in\mathbf{Z}_{\ell}$ of $U(i,N)$. The linear maps of the representation map $v_k$ to $v_{k+1}$ if possible; and to $0$, otherwise. We can draw a picture of the $\Q(4)$-representation $U(2,10)$ as follows, which makes clear the structure of the indecomposables:
\begin{center}
\scalebox{0.6}{
\begin{tikzpicture}[->,>=stealth',shorten >=1pt,auto,node distance=3cm,
  thick,
  node/.style={}]
\node (1) {$\bullet_0$};
\node (1') [left=0.25 of 1] {$\bullet$};
\node (2) [above right=0.5 and 0.5 of 1] {$\bullet_1$};
\node (2') [above=0.25 of 2] {$\bullet$};
\node (2'') [above=0.25 of 2'] {$\bullet$};
\node (3) [below right=0.5 and 0.5 of 2] {$\bullet_2$};
\node (3') [right=0.25 of 3] {$\bullet$};
\node (3'') [right=0.25 of 3'] {$\bullet$};
\node (4) [below left=0.5 and 0.5 of 3] {$\bullet_{3}$};
\node (4') [below=0.25 of 4] {$\bullet$};
\path[->]
(1) edge [bend left=15]  (2')
(2) edge [bend left=15]  (3)
(3) edge [bend left=15]  (4)
(4) edge [bend left=15]  (1)
(2') edge [bend left=15]  (3')
(3') edge [bend left=15]  (4')
(4') edge [bend left=15]  (1')
(1') edge [bend left=15]  (2'')
(2'') edge [bend left=15]  (3'');
\end{tikzpicture}}
\end{center}

Define $\mc{U}(\ell,x) := \{ (i,N) \in \Z_{\ell} \times \mbf{N} \textrm{ such that } | \{ 0 \le j \le N-1, \  j \equiv -i \ \mathrm{mod} \ \ell \} | \le x \}$. 

\begin{theorem}\label{thm:cenc_indecs_circle}
The isomorphism classes of indecomposable representations of $\A(\ell,x)$ are parametrized by the representations 
$$
\{ U(i,N) \ | \ (i,N) \in \mc{U}(\ell,x) \}. 
$$
\end{theorem}
\begin{proof}
The universal covering quiver $\Gamma$ is an infinite quiver of type $\mathsf{A}$. However, the corresponding relations are a bit tricky. They are that 
$$
\varphi_{\ell-1}^{(i)} \circ \cdots \circ \varphi_0^{(i)}\circ \varphi_{\ell-1}^{(i+1)} \circ \cdots \circ \varphi_{\ell-1}^{(i+x)} \circ \cdots \circ \varphi_0^{(i+x)} = 0
$$
for all $i \in \Z$ (and note that the composition of $x\cdot \ell$ maps is always equals $0$, as one might expect). But we can essentially ignore this and note that every indecomposable representation $U_{\Gamma}(i,N)$ of $\Gamma$ is nilpotent (here $i \in \Z$ and $N \ge 1$, and the representation is defined just as for the cyclic quiver). So it suffices to check which of these factors through $\A(\ell,x)$ after applying the covering functor $F : \rep_{\C} \C \Gamma \rightarrow \rep_{\C} \C \Q(\ell)$. We have $F(U_{\Gamma}(i,N)) = U(\overline{i},N)$, where $\overline{i}$ is the image of $i$ in $\Z_{\ell}$. Now we note first that the nilpotent endomorphism $\varphi_{\ell-1} \circ \cdots \circ \varphi_0$ of $U(\overline{i},N)_0$ is a single Jordan block (of size $\dim U(\overline{i},N)_0$). Therefore it factors through $\A(\ell,x)$ if and only if $\dim U(\overline{i},N)_0 \le x$.  But 
\begin{align*}
\dim U(\overline{i},N)_0 & = \dim \bigoplus_{j \in \Z} U_{\Gamma}(i,N)_{j \ell} \\ 
&= | \{ j \in \Z \ |  \ i \le \ell j \le i + N - 1 \} | \\
 & =  | \{ 0 \le k \le N-1 \ |  \ k \equiv -i \ \mathrm{mod} \ \ell \} |
\end{align*}
Thus, the indecomposable representations of $\A(\ell,x)$ that are obtained from $\Gamma(\ell,x)$ via the covering functor are precisely those $U(i,N)$ such that $(i,N) \in \mc{U}(\ell,x)$. Since we know by \cite{Ke} that the algebra $\A(\ell,x)$ is representation-finite, Covering Theory \cite{Ga3} implies that these are, in fact, all isomorphism classes of indecomposables.
\end{proof}

\subsubsection{Classification of indecomposables of dimension vector $(1,*)$}\label{sssect:cenc_indecs1}
The second case deals with indecomposable $\A_{\infty}(\ell,x)$-representations $M$ with $(\dim M)_{\infty} = 1$. The classification is given by Frobenius circle diagrams. 
\begin{theorem}\label{thm:cenc_indecs1}
Fix $x, \ell \ge 1$. 
\begin{enumerate}
\item There are canonical bijections between:
\begin{itemize}
\item The set of isomorphism classes of indecomposable nilpotent representations $M$ of $\Q_{\infty}(\ell)$ with $(\dim M)_{\infty} = 1$. 
\item The set $\Ca_F(\ell)$ of Frobenius circle diagrams. 
\item The set of all partitions. 
\end{itemize}
\item These bijections restrict to bijections between:
\begin{itemize}
\item The set of isomorphism classes of indecomposable representations $M$ of $\A_{\infty}(\ell,x)$ with $(\dim M)_{\infty} = 1$. 
\item The set $\{ C \in \Ca_F(\ell) \ | \ \wt_{\ell}(C) \le x \}$ of Frobenius circle diagrams of weight at most $x$. 
\item The set of all partitions $\{ \lambda \in \mc{P} \ | \ \wt_{\ell}(\lambda) \le x \}$ of weight at most $x$. 
\end{itemize}
\end{enumerate}
\end{theorem}
\begin{proof}
It is clear that statement (2) implies statement (1). We concentrate on statement (2). It has already been explained in section \ref{sec:circle} that the set $\{ C \in \Ca_F(\ell) \ | \ \wt_{\ell}(C) = x \}$ of Frobenius circle diagrams of weight $x$ is in bijection with the set of all partitions of weight $x$. Therefore it suffices to show that the set of isomorphism classes of indecomposable representations $M$ of $\A_{\infty}(\ell,x)$ with $(\dim M)_{\infty} = 1$ is in bijection with the set $\{ C \in \Ca_F(\ell) \ | \ \wt_{\ell}(C) \le x \}$. Let $M$ be an indecomposable representation of $\A_{\infty}(\ell,x)$ with $(\dim M)_{\infty} = 1$. We denote by $U$ the restriction of $M$ to $\Q(\ell)$. By Theorem \ref{thm:cenc_indecs_circle}, we may assume, without loss of generality, that $U = U(i_1,N_1) \oplus \cdots \oplus U(i_k,N_k)$ for some $(i_j,N_j) \in \mc{U}(\ell,x)$. That is, $U$ is described by a certain circle diagram $C$. The embedding of $\C = M_{\infty}$ into $U_0$ defines a marking of the circle diagram: let $v$ be the image of $1 \in \C$ in $U_0$, and recall that the vertices in $b_0$ are a basis of $U_0$. Then a vertex $i$ of $b_0$ is marked if and only if the coefficient of the corresponding basis vector in the expansion of $v$ is non-zero. The fact that the indecomposable representations correspond precisely to  Frobenius circle diagrams can then be shown by repeating the arguments given in the proof of Lemma \ref{lem:enc_indec1n}. We do not repeat them here. 
\end{proof}
Given a Frobenius circle diagram $C$, we denote by $M_{C}$ the corresponding canonical indecomposable nilpotent representation.

\subsection{A combinatorial parametrization}\label{sec:comborbits}

Given a fixed dimension vector $\mbf{d}_{\infty}=(1,d_0,\dots,d_{\ell-1})$ of $\Q_{\infty}(\ell)$, we denote $\mbf{d}:=(d_{0},\dots,d_{\ell-1})$ and we can parametrize the $G$-orbits in the cyclic enhanced nilpotent cone by making use of section \ref{ssect:ecnc_indec11}. Denote by $\Cal_{2,F}(\mbf{d})$ the set of tuples $(\Cal_F,\Cal)$ of a Frobenius circle diagram $\Cal_F$ with dimension vector $\mbf{d}_1$ and a circle diagram $\Cal$ with dimension vector $\mbf{d}_2$, such that $\mbf{d}_1+\mbf{d}_2=\mathbf{d}$. Then the results of section \ref{ssect:ecnc_indec11} imply that:

\begin{theorem}\label{thm:cenc_class}
Every representation $M$ in $\Rep(\A_{\infty}(\ell,x),\mbf{d}_{\infty})$ decomposes as a direct sum 
\begin{equation}\label{eq:Mbipart2}
M\simeq M_{C'}\oplus\bigoplus_{i=1}^{\ell(C)} U(0,C_i)
\end{equation}
of indecomposable representations, for some unique tuple $(C',C) \in\Cal_{2,F}(\mbf{d})$. 
\end{theorem}

Recall that the $\GL_{\mathbf{d}_{\infty}}$-orbits in $\Rep(\A_{\infty}(\ell,x),\mbf{d}_{\infty})$ are the same as the $\GL_{\mathbf{d}}$-orbits. We denote the $\GL_{\mathbf{d}}$-orbit of the representation (\ref{eq:Mbipart2}) by $\mc{O}_{C',C}$. We deduce from Theorem \ref{thm:cenc_class} that:

\begin{proposition}
There is a bijection $\Phi_{\ell}: (C',C)\mapsto \mc{O}_{C',C}$ from $\Cal_{2,F}(\mbf{d})$ to the set of $\GL_{\mathbf{d}}$-orbits in $\Rep(\A_{\infty}(\ell,x),\mbf{d}_{\infty})$.
\end{proposition} 

In applications to admissible $\dd$-modules \cite{BeB2}, we are interested in representations of a very particular dimension. Namely, if $\delta = (1, \ds, 1)$ is the minimal imaginary root for $\Q(\ell)$ as in section \ref{sec:affineA}, then we fix 
$$
\vdim := \varepsilon_{\infty} + n \delta = (1, n, \ds, n). 
$$
We use Theorem \ref{thm:cenc_indecs_circle} and Theorem \ref{thm:cenc_indecs1} to derive a combinatorial enumeration of the $G$-orbits in the enhanced cyclic nilpotent cone $\N_{\infty}(\ell,n) := \Rep(\A_{\infty}(\ell),\vdim)$. Recall from Theorem \ref{thm:cenc_indecs1} that the indecomposable nilpotent representations $M$ of $\Q_{\infty}(\ell)$ with $(\dim M)_{\infty} = 1$ are parametrized by the set of partitions. Given a partition $\lambda$, we write $M_{\lambda}$ for the corresponding indecomposable representation and $\mbf{d}_{\lambda}$ for its dimension vector. Recall from section \ref{sec:affineA} the definition of $\ell$-residue, $\res_{\ell}(\lambda)$, of a partition $\lambda$ and the shifted $\ell$-residue, $\sres_{\ell}(\nu)$, of an $\ell$-multipartition $\nu$.  

\begin{proposition}\label{prop:cenc_param_orbits}
The $\GL_{\mathbf{d}}$-orbits in the enhanced cyclic nilpotent cone $\N_{\infty}(\ell,n)$ are naturally labelled by the set 
$$
\mathcal{Q}(n,\ell) := \left\{ (\lambda;\nu) \in \mathcal{P} \times \mathcal{P}_{\ell} \ | \ \res_{\ell}(\lambda) + \sres_{\ell}(\nu) = n \delta\right\}. 
$$
\end{proposition}

\begin{proof}
Let $\mathcal{O}$ be an orbit and $M \in \mathcal{O}$. Then $M = M_{\lambda} \oplus Y$, where $M_{\lambda}$ is an indecomposable nilpotent representations $M$ of $\Q_{\infty}(\ell)$ with $(\dim M)_{\infty} = 1$ and $Y = U(i_1, N_1) \oplus \cdots \oplus U(i_k,N_k)$ is a direct sum of indecomposable nilpotent representations of $\Q(\ell)$ such that 
$$
\dim M_{\lambda} + \sum_{j = 1}^k \dim U(i_j,N_j) = \varepsilon_{\infty} + n \delta. 
$$
We associate to $Y_{\nu} := Y$ the multipartition $\nu$, where $\nu^{(i)} = N_{j_1} \ge N_{j_2} \ge \dots$, where the $j_r$ run over all $1 \le j_r \le k$ such that $i_{j_r} = i$. Thus, the question is simply to find all $(\lambda;\nu) \in \mathcal{P} \times \mathcal{P}_{\ell}$ such that $\dim (M_{\lambda} \oplus Y_{\nu}) = e_{\infty} + n \delta$. Under the identification of $\Z \Q(\ell)$ with $\Z[\Z_{\ell}]$, we have $\dim  U(i,N) = \sigma^i \res_{\ell}(N)$ and hence
$$
\dim Y =  \sres_{\ell}(\nu).
$$
Therefore it suffices to show that $\mbf{d}_{\lambda}= \dim M_{\lambda}$ equals $\varepsilon_{\infty} + \res_{\ell}(\lambda)$. Recall that $\Gamma_{\infty}$ is the covering quiver of $\Q_{\infty}(\ell)$ and $\Gamma$ the covering quiver of $\Q(\ell)$. As in the proof of Theorem \ref{thm:cenc_indecs_circle}, let $F : \rep_{\C} \C \Gamma \rightarrow \rep_{\C} \C \Q(\ell)$ denote the covering functor. If $M_{\lambda}'$ is the unique lift of $M_{\lambda}$ to $\Gamma_{\infty}$ such that $v^{(i)} = 0$ for all $i \neq 0$ (see section \ref{sec:reptype}), then $M_{\lambda}' |_{\Gamma}$ equals $U_{\Gamma}(-b_1,p_1) \oplus \cdots \oplus U_{\Gamma}(-b_k,p_k)$, where $(a_1 > \cdots > a_r ; b_1 > \cdots > b_r)$ is $\lambda$ written in Frobenius form and $p_i := a_i + b_i + 1$. Therefore,
$$
\dim M_{\lambda} = \sum_{i =1}^r \sigma^{-b_i} \res_{\ell}(p_i) = \res_{\ell}(\lambda)
$$
as required. 
\end{proof}

In the proof of Proposition \ref{prop:cenc_param_orbits} we have shown that 
\begin{equation}\label{eq:dimension}
\dim  U(i,N) = \sigma^i \res_{\ell}(N), \quad \mbf{d}_{\lambda} = \varepsilon_{\infty} + \res_{\ell}(\lambda).
\end{equation}
If $\nu = \emptyset$ then $(\lambda;\nu) \in \mathcal{Q}(n,\ell)$ if and only if $\res_{\ell}(\lambda) = n \delta$. The set of all such $\lambda$ is precisely the set of partitions of $n \ell$ that have trivial $\ell$-core. This set, in turn, is in bijection with the set $\mc{P}_{\ell}(n)$ of $\ell$-multipartitions of $n$, the bijection given by taking the $\ell$-quotient of $\lambda$ i.e. if $\lambda$ has trivial $\ell$-core then it is uniquely defined by its $\ell$-quotient.

\subsection{Translating between the different parametrizations}\label{ssect:cenc_transl}

The goal of this final section is to describe how to pass directly between the combinatorial parametrization of the $\GL_{\mathbf{d}}$-orbits in the enhanced cyclic nilpotent cone given by Johnson \cite{Joh}, and our parametrization given in Proposition \ref{prop:cenc_param_orbits}. In order to do this, we first recall the former.\\[1ex]

A tuple $(\lambda,\epsilon)$ is called an \textit{$\ell$-coloured partition} if $\lambda=(\lambda_1,...,\lambda_k)\in\Pa$ and $\epsilon=(\epsilon_1,...,\epsilon_k)\in(\mathbf{Z}/\ell\mathbf{Z})^k$. This coloured partition gives rise to a \textit{coloured Young diagram} $Y(\lambda,\epsilon)$ by defining the \textit{colour of the box}  $(i,j)$ of $Y(\lambda)$ to be $\chi(i,j):=\epsilon_i+[\lambda_i-j]$; where $1\leq i\leq \ell(\lambda), 1\leq j\leq \lambda_i$ and $[x]$ denotes the residue class of $x$ modulo $\ell$. Its \textit{signature} is defined to be $\xi(\lambda,\epsilon)=(\xi(\lambda,\epsilon)_m)_{0 \leq m\leq \ell-1}$ and 
$$
\xi(\lambda,\epsilon)_m:= |\{(i,j) \in Y(\lambda) \ | \ \epsilon_i+[\lambda_i-j]=m \}|. 
$$
In this language, the well-known classification of orbits in the cyclic nilpotent cone can be stated as: 

\begin{lemma} Let $\mbf{d}$ be a dimension vector of $\A(\ell)$. 
There is a bijection from the set of $\ell$-coloured partitions of signature $\mbf{d}$ to the isomorphism classes of nilpotent $\A(\ell)$-representations of dimension vector $\mbf{d}$.
\end{lemma}
Given an $\ell$-coloured partition  $(\lambda,\epsilon)$ of signature $\mbf{d}$, it is mapped to the $\A(\ell)$-representation $(V,N)$ where $V=V_0 \oplus \cdots \oplus V_{\ell-1}$ has a coloured Jordan basis $\{v_{i,j}\}_{1\leq i\leq l(\lambda), 1\leq j\leq \lambda_i}$, that is, $v_{i,j}$ is a basis vector of $V_{\chi(i,j)}$. Furthermore, $Nv_{i,j}=0$ if $j=1$ and $Nv_{i,j}=v_{i,j-1}$ if $j>1$ and the basis can, thus, be depicted best by the coloured Young diagram $Y(\lambda,\epsilon)$.\\[1ex]

Note that this parametrization can be directly translated to our circle diagrams of Theorem \ref{thm:cenc_indecs_circle}: The circle diagram consists of $\ell(\lambda)$ circles of which the $i$-th starts in vertex $\epsilon_i$ and is of length $\lambda_i$. We denote this circle diagram by $C(\lambda,\epsilon)$, it corresponds to the representation 
\[\bigoplus_{1\leq i\leq \ell(\lambda)} U(\epsilon_i,\lambda_i). \]

Let us call a tuple $(\lambda,\epsilon,\nu)$ a \textit{marked coloured partition} if $(\lambda,\epsilon)$ is a coloured partition and $\nu:\mathbf{N}\rightarrow \mathbf{Z}$ is a \textit{marking function}, which satisfies $\nu_i\leq \lambda_i$ for all $i$. We define $\mu:=(\mu_i)_{1\leq i\leq \ell(\lambda)}=(\lambda_i-\nu_i)_{1\leq i\leq \ell(\lambda)}$. Note that we have switched the roles of $\mu$ and $\nu$ in comparison to \cite{Joh} - this is consistent with our conventions in subsection \ref{ssect:enc_transl}. A marked coloured partition $(\lambda,\epsilon,\nu)$ is called a \textit{striped $\ell$-bipartition}, if 
\begin{enumerate}
\item $\epsilon+[\lambda-\nu] = 0$ in $\mathbb{Z} / \ell \mathbb{Z}$, 
\item
$-\ell<\nu_i$ for all $i$,
\item
$\nu_j<\nu_i+\ell$ and $\mu_j<\mu_i+\ell$ for each $i<j$.
\end{enumerate}
In the case $\ell=1$, this yields the set of double partitions $\Pa_2(n)$, where $n$ is the dimension vector, since $(\mu,\nu)$ is a bi-partition of $n$. The set of all striped $\ell$-bipartitions is denoted $\Pa_{st}(\ell)$; the subset with fixed signature $\xi$ is denoted $\Pa_{st}(\ell,\xi)$. The classification of orbits in the enhanced cyclic nilpotent cone, as in \cite{Joh}, is then given by:

\begin{proposition}
There is a bijection $\Xi_{\ell}$ from the set $\Pa_{st}(\ell)$ to the set of orbits in the enhanced cyclic nilpotent cone. This bijection restricts to fixed dimension vectors, i.e. $\Pa_{st}(\ell, \mbf{d})$ is in bijection with the $\GL_{\mathbf{d}}$-orbits in $\Rep(\A_{\infty}(\ell),\mbf{d}_{\infty})$.
\end{proposition}
We can write down $\Xi_{\ell}$ explicitly . Let  $(\lambda,\epsilon,\nu)\in\Pa_{st}(\ell,\mbf{d})$, then there is a coloured Jordan basis $B:=\{v_{i,j}\}_{1\leq i\leq l(\lambda), 1\leq j\leq \lambda_i}$ and a nilpotent $\A(\ell)$-representation $N$ in normal-form adapted to the basis $B$ as described above. Set $v_{i,j}=0$ if $j\leq 0$. Then	$\Xi_{\ell}(\lambda,\epsilon,\nu)$ is defined to be the orbit of the nilpotent representation $(v,N)$ of the cyclic enhanced nilpotent cone, where $v = \sum_{i=1}^{\ell(\lambda)} v_{i,\nu_i}$. Pictorially, this means that the $i$-th circle is marked at position $\mu_i$ (not $\nu_i$, as one might expect). Interpreting $\Xi_{\ell}(\lambda,\epsilon,\nu)$ as an  $\A_{\infty}(\ell)$-representation, we obtain 

\[\begin{tikzpicture}
\matrix (m) [matrix of math nodes, row sep=0.05em,
column sep=2em, text height=1.5ex, text depth=0.2ex]
{\Xi_{\ell}(\lambda,\epsilon,\nu)\cong & \C & \C^n\\ };
\path[->]
(m-1-2) edge node[above=0.05cm] {$\iota$} (m-1-3)
(m-1-3) edge [loop right] node{$C(\lambda,\epsilon)$} (m-1-3);\end{tikzpicture}, \]
where $n=\sum_{i=0}^{\ell-1}\mbf{d}_i$ and the right hand part is given by a circle diagram and, technically speaking, a graded vector space with a cycle of maps.  Furthermore,  $\iota=\sum_{i=1}^{\ell(\lambda)} e_{i,\nu_i}$ and $e_{i,j}$ is the standard embedding into the $(i,j)$-th basis vector $v_{i,j}$ of $B$. \\[1ex]

Our aim is to define a map which translates this parametrization to our classification $\Phi_{\ell}$ from Proposition \ref{prop:cenc_param_orbits}. That is, we have to define a bijection $\Psi_{\ell}:\Pa_{st}(\ell,\mbf{d})\rightarrow\Cal_{2,F}(\mbf{d})$, such that
$$
\Phi_{\ell} \circ \Psi_{\ell} (\lambda,\epsilon,\nu) = \mc{O}_{\Psi_{\ell}(\lambda,\epsilon,\nu)} = \Xi_{\ell} (\lambda,\epsilon,\nu). 
$$
We define $\Psi_{\ell}$ on $\Pa_{st}(\ell)$, since the restriction to a fixed signature will then obviously determine the desired bijection.\\[1ex]

Let $(\lambda,\epsilon,\nu) \in \Pa_{st}(\ell)$ and $\mu_i=\lambda_i-\nu_i$. Then $i$ belongs to the set $\mathrm{Re}(\lambda,\epsilon,\nu) \subset \{ 1, \ds, \ell(\lambda) \}$ of removable rows if and only if one of the following holds: 
\begin{itemize}
\item[(a)] $\nu_i \leq 0$,
\item[(b)] there is some $j>i$ such that $\mu_j \geq \mu_{i}$ and $(\lambda_i,\epsilon_i)\neq (\lambda_j,\epsilon_j)$ or
\item[(c)] there is some $j<i$, such that  $\nu_{j} \leq \nu_i$.
\end{itemize}

In case $\ell=1$,  this reduces to our definition of removable rows in section \ref{ssect:enc_transl}. The set $\mathrm{Re}(\lambda,\epsilon,\nu)$ leads to a coloured partition as follows. Define the partition $\hat{\lambda}=(\lambda_i)_{i\in \mathrm{Re}(\lambda,\epsilon,\nu)}$ and the colouring $\hat{\epsilon}=(\epsilon_i)_{i\in \mathrm{Re}(\lambda,\epsilon,\nu)}$. Then $(\hat{\lambda},\hat{\epsilon})$ is a coloured partition and $C(\hat{\lambda},\hat{\epsilon})$ is a circle diagram. The remaining part of $(\lambda,\epsilon,\nu)$ determines a Frobenius circle diagram: the circles are parametrized by $I:=\{1\leq i\leq \ell(\lambda)\mid i\notin \mathrm{Re}(\lambda,\epsilon,\nu)\}$. The $i$-th circle is of length $\lambda_i$, starts in position $\epsilon_i$ and is marked at the $\mu_i$-th vertex (clockwise). By construction, the mark is always in position $0$ and the circles form a Frobenius partition. Let us denote this Frobenius circle diagram $C_F(\lambda,\epsilon,\nu)$. We define $\Psi_{\ell}(\lambda,\epsilon,\nu):=(C(\hat{\lambda},\hat{\epsilon}), C_F(\lambda,\epsilon,\nu))$.

\begin{theorem}\label{thm:johnson}
The map $\Psi_{\ell}$ is a bijection such that 
$$
\Phi_{\ell} \circ \Psi_{\ell} (\lambda,\epsilon,\nu) = \mc{O}_{\Psi_{\ell}(\lambda,\epsilon,\nu)} = \Xi_{\ell} (\lambda,\epsilon,\nu). 
$$
\end{theorem}

\begin{proof}
We check that $\Phi_{\ell} \circ \Psi_{\ell}(\lambda,\epsilon,\nu) = \Xi_{\ell}(\lambda,\epsilon,\nu)$ for all $(\lambda,\epsilon,\nu)$. This implies that $\Psi_{\ell}$ is a bijection.  Let $M = (v,X) \in \Xi_{\ell}(\lambda,\epsilon,\nu)$. This means that there is a coloured normal basis $B=\{ v_{i,j} \}$ of $V$ such that $X$ is of cyclic normal form  $(\lambda,\epsilon)$ and $v = \sum_{i = 1}^{\ell(\lambda)} v_{i,\nu_i}$. We wish to show that $M$ decomposes into  
$$
M\simeq M' \oplus\bigoplus_{i\in \mathrm{Re}(\lambda,\epsilon,\nu)} U(\epsilon_i,\lambda_i)
$$
where $M'$ is the representation corresponding to  $C_F(\lambda,\epsilon,\nu)$ via Theorem \ref{thm:cenc_indecs1}.\\[1ex]

 The proof is by induction on $\ell(\lambda)$. Assume that there exists $i \in \{1 ,\ds, \ell(\lambda) \}$, such that there is $j>i$ with $\mu_i < \mu_{j}$ and $(\lambda_i,\epsilon_i)\neq (\lambda_j,\epsilon_j)$. Then, repeating the argument in the proof of  Lemma \ref{lem:enc_indec1n}, we deduce that $M \simeq M_1 \oplus U(\epsilon_i,\lambda_i)$, where $M_1\in\Xi(\lambda',\epsilon',\nu')$, and $(\lambda',\epsilon',\nu')$ is obtained from $(\lambda',\epsilon',\nu')$ by removing the $i$-th component of $\lambda$, $\epsilon$ and $\nu$. By induction, $\Xi_{\ell}(\lambda',\epsilon',\nu') =  \Phi_{\ell} \circ \Psi_{\ell} (\lambda',\epsilon',\nu')$ and hence $\Xi_{\ell}(\lambda,\epsilon,\nu) =  \Phi_{\ell} \circ \Psi_{\ell}(\lambda,\epsilon,\nu)$. We proceed in exactly the same way if $\nu_{j} \leq \nu_i$ for $j<i$, or if $\nu_i\leq 0$: Then, $M$ has a direct summand isomorphic to $U(\epsilon_i,\lambda_i)$. We have shown
 $$
M\simeq M' \oplus\bigoplus_{i\in \mathrm{Re}(\lambda,\epsilon,\nu)} U(\epsilon_i,\lambda_i)
$$
The representation $M'$ is the indecomposable representation belonging to the Frobenius circle diagram $C_F(\lambda,\epsilon,\nu)$. Thus, it is indecomposable and we have decomposed $\Xi_{\ell} (\lambda,\epsilon,\nu)$ into a direct sum of indecomposables:
$$
\Phi_{\ell} \circ\Psi_{\ell} (\lambda,\epsilon,\nu)= \Phi_{\ell}(C(\hat{\lambda},\hat{\epsilon}), C_F(\lambda,\epsilon,\nu))=\Xi_{\ell} (\lambda,\epsilon,\nu).
$$
This completes the proof of the theorem. 
\end{proof}

We end this section by giving an example.
\begin{example}
Consider the $4$-striped bi-partition $(\lambda,\epsilon,\mu)$, where $\lambda=(16,14,13,11,9,6,5,5,2)$, $\epsilon=(0,2,0,1,3,0,2,2,0)$ and $\mu=(8,4,5,4,0,2,3,3,-2)$. It can be depicted as follows; the position of $\mu_i$ is highlighted:
\begin{center}

$ 
\scalebox{0.7}{ \begin{ytableau}
3&2&1&0&3&2&1&*(lightblue)0&3&2&1&0&3&2&1&0\\
\none&\none&\none&\none&3&2&1&*(lightblue)0&3&2&1&0& 3&2&1&0&3&2\\
\none&\none&\none&0&3&2&1&*(lightblue)0& 3&2&1&0&3&2&1&0\\
\none&\none&\none&\none&3&2&1&*(lightblue)0& 3&2&1&0&3&2&1\\
\none&\none&\none&\none&\none&\none&\none&\none&3&2&1&0& 3&2&1&0&3 \\
\none&\none&\none&\none&\none&\none&1&*(lightblue)0& 3&2&1&0\\
\none&\none&\none&\none&\none&2&1&*(lightblue)0& 3&2\\
\none&\none&\none&\none&\none&2&1&*(lightblue)0& 3&2\\
\none&\none&\none&\none&\none&\none&\none&\none&\none&\none&1 &0\\
\end{ytableau}}$\end{center}
Then  $\mathrm{Re}(\lambda,\epsilon,\mu)=(2,3,5,6,8,9)$. 
The remaining rows  yield a Frobenius circle diagram in the obvious way:
\begin{center}

$ 
\scalebox{0.7}{ \begin{ytableau}
3&2&1&0&3&2&1&*(lightblue)0&3&2&1&0&3&2&1&0\\
\none&\none&\none&\none&3&2&1&*(lightblue)0& 3&2&1&0&3&2&1\\
\none&\none&\none&\none&\none&2&1&*(lightblue)0& 3&2\\
\end{ytableau}}$\end{center}

The rows which correspond to $\mathrm{Re}(\lambda,\epsilon,\mu)$ are removed and yield a circle diagram, where the circles start in (highlighted) positions $\epsilon_i$  and have lengths $\lambda_i$:
\begin{center}

$ 
\scalebox{0.7}{ \begin{ytableau}
\none&\none&\none&\none&3&2&1&0&3&2&1&0& 3&2&1&0&3&*(lightred)2\\
\none&\none&\none&\none&\none&0&3&2&1&0& 3&2&1&0&3&2&1&*(lightred)0\\
\none&\none&\none&\none&\none&\none&\none&\none&\none&3&2&1&0& 3&2&1&0&*(lightred)3 \\
\none&\none&\none&\none&\none&\none&\none&\none&\none&\none&\none&\none&1&0& 3&2&1&*(lightred)0\\
\none&\none&\none&\none&\none&\none&\none&\none&\none&\none&\none&\none&\none&2&1&0& 3&*(lightred)2\\
\none&\none&\none&\none&\none&\none&\none&\none&\none&\none&\none&\none&\none&\none&\none&\none&1 &*(lightred)0\\
\end{ytableau}}$\end{center}
The direct sum of indecomposable representations of the cyclic enhanced nilpotent cone can, thus, be read of directly.
\end{example}

\bibliographystyle{plain}

\begin{thebibliography}{10}

\bibitem{AH}
Pramod~N. Achar and Anthony Henderson.
\newblock Orbit closures in the enhanced nilpotent cone.
\newblock {\em Adv. Math.}, 219(1):27--62, 2008.

\bibitem{ASS}
Ibrahim Assem, Daniel Simson, and Andrzej Skowro{\'n}ski.
\newblock {\em Elements of the representation theory of associative algebras.
  {V}ol. 1}, volume~65 of {\em London Mathematical Society Student Texts}.
\newblock Cambridge University Press, Cambridge, 2006.
\newblock Techniques of representation theory.

\bibitem{BeB2}
Gwyn Bellamy and Magdalena Boos.
\newblock Semi-simplicity of the category of admissible D-modules.
\newblock Preprint, https://arxiv.org/abs/1709.08986, 2017.

\bibitem{Ber}
Joseph~N. Bernstein.
\newblock {$P$}-invariant distributions on {${\rm GL}(N)$} and the
  classification of unitary representations of {${\rm GL}(N)$}
  (non-{A}rchimedean case).
\newblock In {\em Lie group representations, {II} ({C}ollege {P}ark, {M}d.,
  1982/1983)}, volume 1041 of {\em Lecture Notes in Math.}, pages 50--102.
  Springer, Berlin, 1984.

\bibitem{Bo4}
Klaus Bongartz.
\newblock Algebras and quadratic forms.
\newblock {\em J. London Math. Soc. (2)}, 28(3):461--469, 1983.

\bibitem{Bo1}
Klaus Bongartz.
\newblock Minimal singularities for representations of {D}ynkin quivers.
\newblock {\em Comment. Math. Helv.}, 69(4):575--611, 1994.

\bibitem{Bo2}
Klaus Bongartz.
\newblock On degenerations and extensions of finite-dimensional modules.
\newblock {\em Adv. Math.}, 121(2):245--287, 1996.

\bibitem{B2}
Magdalena Boos.
\newblock Finite parabolic conjugation on varieties of nilpotent matrices.
\newblock {\em Algebr. Represent. Theory}, 17(6):1657--1682, 2014.

\bibitem{BdlPS}
Thomas Br{\"u}stle, Jos{\'e}~Antonio de~la Pe{\~n}a, and Andrzej
  Skowro{\'n}ski.
\newblock Tame algebras and {T}its quadratic forms.
\newblock {\em Adv. Math.}, 226(1):887--951, 2011.

\bibitem{DlPS}
Jos{\'e}~Antonio de~la Pe{\~n}a and Andrzej Skowro{\'n}ski.
\newblock The {T}its forms of tame algebras and their roots.
\newblock In {\em Representations of algebras and related topics}, EMS Ser.
  Congr. Rep., pages 445--499. Eur. Math. Soc., Z\"urich, 2011.

\bibitem{Dr}
Yuriy~A. Drozd.
\newblock Tame and wild matrix problems.
\newblock In {\em Representation theory, {II} ({P}roc. {S}econd {I}nternat.
  {C}onf., {C}arleton {U}niv., {O}ttawa, {O}nt., 1979)}, volume 832 of {\em
  Lecture Notes in Math.}, pages 242--258. Springer, Berlin, 1980.

\bibitem{DoGinT}
V.~Ginzburg G.~Dobrovolska and R.~Travkin.
\newblock Moduli spaces, indecomposable objects and potentials over a finite
  field.
\newblock Preprint, {\em arXiv}, 1612.01733v1, 2016.

\bibitem{Ga3}
Peter Gabriel.
\newblock The universal cover of a representation-finite algebra.
\newblock In {\em Representations of algebras ({P}uebla, 1980)}, volume 903 of
  {\em Lecture Notes in Math.}, pages 68--105. Springer, Berlin, 1981.

\bibitem{Joh}
Casey~Patrick Johnson.
\newblock {\em Enhanced nilpotent representations of a cyclic quiver}.
\newblock ProQuest LLC, Ann Arbor, MI, 2010.
\newblock Thesis (Ph.D.)--The University of Utah.

\bibitem{Ke}
Gisela Kempken.
\newblock {\em Eine {D}arstellung des {K}\"ochers {$\~A_{k}$}}.
\newblock Bonner Mathematische Schriften [Bonn Mathematical Publications], 137.
  Universit\"at Bonn, Mathematisches Institut, Bonn, 1982.
\newblock Dissertation, Rheinische Friedrich-Wilhelms-Universit{\"a}t, Bonn,
  1981.

\bibitem{MautnerPaving}
C.~Mautner.
\newblock Affine pavings and the enhanced nilpotent cone.
\newblock {\em Proc. Amer. Math. Soc.}, 145(4):1393--1398, 2017.

\bibitem{NoeSk}
Rainer N{\"o}renberg and Andrzej Skowro{\'n}ski.
\newblock Tame minimal non-polynomial growth simply connected algebras.
\newblock {\em Colloquium Mathematicae}, 73(2):301--330, 1997.

\bibitem{Se}
Jean-Pierre Serre.
\newblock Espaces fibr{\'e}s alg{\'e}briques (d'apr{\`e}s {A}ndr{\'e} {W}eil).
\newblock In {\em S{\'e}minaire {B}ourbaki, {V}ol.\ 2}, pages Exp.\ No.\ 82,
  305--311. Soc. Math. France, Paris, 1995.

\bibitem{Sko2}
Andrzej Skowro{\'n}ski.
\newblock Simply connected algebras and {H}ochschild cohomologies [ {MR}1206961
  (94e:16016)].
\newblock In {\em Representations of algebras ({O}ttawa, {ON}, 1992)},
  volume~14 of {\em CMS Conf. Proc.}, pages 431--447. Amer. Math. Soc.,
  Providence, RI, 1993.

\bibitem{Sko3}
Andrzej Skowro{\'n}ski.
\newblock Tame algebras with strongly simply connected galois coverings.
\newblock {\em Colloquium Math.}, 72(2):335--351, 1997.

\bibitem{SunEnhancednilcone}
M.~Sun.
\newblock Point stabilisers for the enhanced and exotic nilpotent cones.
\newblock {\em J. Group Theory}, 14(6):825--839, 2011.

\bibitem{Tr}
Roman Travkin.
\newblock Mirabolic {R}obinson-{S}chensted-{K}nuth correspondence.
\newblock {\em Selecta Math. (N.S.)}, 14(3-4):727--758, 2009.

\bibitem{Un}
Luise Unger.
\newblock The concealed algebras of the minimal wild, hereditary algebras.
\newblock {\em Bayreuth. Math. Schr.}, (31):145--154, 1990.

\bibitem{Vinberg}
{\`E}.~B. Vinberg.
\newblock The {W}eyl group of a graded {L}ie algebra.
\newblock {\em Izv. Akad. Nauk SSSR Ser. Mat.}, 40(3):488--526, 709, 1976.

\bibitem{Zw1}
Grzegorz Zwara.
\newblock Degenerations for modules over representation-finite algebras.
\newblock {\em Proc. Amer. Math. Soc.}, 127(5):1313--1322, 1999.

\bibitem{Zw2}
Grzegorz Zwara.
\newblock Degenerations of finite-dimensional modules are given by extensions.
\newblock {\em Compositio Math.}, 121(2):205--218, 2000.

\end{thebibliography}

\appendix
\section{Auslander-Reiten quiver}

\subsection[The case (2,1)]{$(\ell,x)=(2,1)$}\label{ssect:arq21}
The Auslander-Reiten quiver for $\Gamma_{\infty}(2,1)$  is given by
\begin{center}
\includegraphics[trim=195 540 240 170, clip,width=220pt]{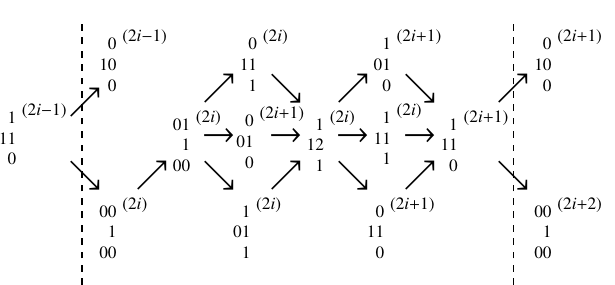}
\end{center}

\subsection[The case (3,1)]{$(\ell,x)=(3,1)$}\label{ssect:arq31}
The Auslander-Reiten quiver for $\Gamma_{\infty}(3,1)$  is given by
\begin{center}
\includegraphics[trim=110 500 200 165, clip,width=320pt]{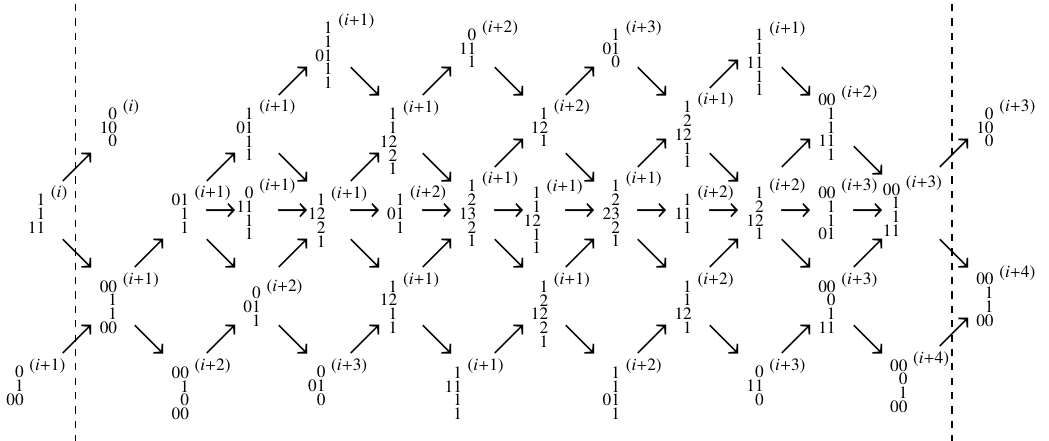}
\end{center}

\end{document}